\documentclass[12pt]{amsart}
\usepackage{cite}
\usepackage[all]{xy}
\usepackage{graphicx}
\usepackage{textcomp}
\usepackage{charter}
\usepackage[vflt]{floatflt}
\usepackage{amssymb}

\theoremstyle{plain}
\newtheorem{thm}{Theorem}[subsection]
\newtheorem{cor}[thm]{Corollary}
\newtheorem{lem}[thm]{Lemma}
\newtheorem{prop}[thm]{Proposition}

\theoremstyle{definition}
\newtheorem{dfn}[thm]{Definition}

\theoremstyle{remark}
\newtheorem{rem}[thm]{Remark}

\newtheorem{ex}[thm]{Example}

\numberwithin{equation}{section}

\newcommand{\pd}[2]
{\dfrac{\partial #1}{\partial #2}}
\newcommand{\pda}[1]{\pd{#1}{a}}
\newcommand{\pdb}[1]{\pd{#1}{b}}
\newcommand{\pdt}[1]{\pd{#1}{t}}

\newcommand{\tr}{\operatorname{tr}}

\newcommand{\fg}{{\mathfrak g}}
\newcommand{\cfg}{{\mathfrak g}_{\mathbb{C}}}
\newcommand{\fh}{{\mathfrak h}}
\newcommand{\fG}{{\mathfrak G}}
\newcommand{\fV}{{\mathfrak V}}
\newcommand{\fU}{{\mathfrak U}}
\newcommand{\fH}{{\mathfrak H}}

\newcommand{\fC}{{\mathfrak C}}

\newcommand{\RR}{{\mathbb R}}
\newcommand{\CC}{{\mathbb C}}
\newcommand{\ZZ}{{\mathbb Z}}

\newcommand{\bu}{{\mathbf u}}

\newcommand{\bx}{{\mathbf x}}
\newcommand{\by}{{\mathbf y}}
\newcommand{\bS}{{\mathbf S}}

\newcommand{\cE}{{\mathcal E}}
\newcommand{\cH}{{\mathcal H}}
\newcommand{\cP}{{\mathcal P}}

\newcommand{\cL}{{\mathcal L}}

\newcommand{\cV}{{\mathcal V}}

\renewcommand{\gg}{\gamma}
\newcommand{\ggst}{{\gamma_*}}
\newcommand{\gs}{\sigma}
\newcommand{\gr}{\rho}

\newcommand{\genmatrix}[4]{\left(\begin{array}{cc} #1 & #2 \\ #3 & #4 \end{array}\right)}
\newcommand{\Gmatrix}[2]{\genmatrix{#1}{#2}{0}{1}}
\newcommand{\gmatrix}[2]{\genmatrix{#1}{#2}{0}{0}}

\newcommand{\bcE}{{\boldsymbol \cE}}

\newcommand{\chpinf}{\cH_{\pi,\infty}}
\newcommand{\chpminf}{\cH_{\pi,-\infty}}
\newcommand{\ihptens}{I_{\cV^*} \otimes}
\newcommand{\chpinfplus}{\cH_{\pi_{+},\infty}}
\newcommand{\chpminfplus}{\cH_{\pi_{+},-\infty}}

\newcommand{\expmap}{\operatorname{exp}}

\newcommand{\ad}{\operatorname{ad}}

\begin{document}
\title{Representations of Lie Algebras by non-Skewselfadjoint Operators in Hilbert Space}
\author{Eli Shamovich}
\address{Department of  Mathematics\\ 
Ben-Gurion University of the Negev\\
84105 Beer-Sheva, Israel}
\email{shamovic@math.bgu.ac.il}
\thanks{The research of E. S. was partially carried out during the visits to the Department of Mathematics and Statistics of the University of Konstanz, supported by the EDEN Erasmus Mundus program (30.12.2013 - 30.6.2014) and to the MFO, supported by the Leibnitz graduate student program (6.4.2014-12.4.2014). The research of E. S. was also supported by the Negev fellowship of the Kreitman school of the Ben Gurion University of the Negev.}
\author{Victor Vinnikov}
\address{Department of  Mathematics\\ 
Ben-Gurion University of the Negev\\
84105 Beer-Sheva, Israel}
\email{vinnikov@math.bgu.ac.il}
\begin{abstract}
We study non-selfadjoint representations of a finite dimensional real Lie algebra $\fg$. 
To this end we embed a non-selfadjoint representation of $\fg$ into a more complicated structure,
that we call a $\fg$-operator vessel and that is associated to
an overdetermined linear conservative input/state/output
system on the corresponding simply connected Lie group $\fG$. 
We develop the frequency domain theory of the system in terms of representations of $\fG$,
and introduce the joint characteristic function of a $\fg$-operator vessel which is the analogue
of the classical notion of the characteristic function of a single non-selfadjoint operator. 
As the first non-commutative example, 
we apply the theory to the Lie algebra of the $ax+b$ group, the group of affine transformations of the line.
\end{abstract}
\maketitle
\tableofcontents
\section*{Introduction}
\addtocontents{toc}{\protect\setcounter{tocdepth}{-1}}
\subsection*{Motivation}

Spectral analysis of a non-selfadjoint operator was developed by M.S. Liv\v{s}ic and his collaborators in
the 1940s and the 1950s, see the pioneering paper \cite{Liv1946} and the survey  \cite{BrLiv},
as well as the book \cite{Br}.
In their work, one embeds (implicitly or explicitly) a non-selfadjoint operator $A$ on a Hilbert space $\cH$ into a somewhat more
complicated structure called an operator colligation (or an operator node). An operator colligation is a collection of spaces and operators $\left(A, \cH, \cE, \Phi, \sigma \right)$. Here $\cE$ is an auxiliary Hilbert space, $\Phi \colon \cH \to \cE$ is a bounded linear operator and $\sigma$ is a bound selfadjoint operator on $\cE$ that satisfies the so-called colligation condition:
\[
A - A^* = i \Phi^* \sigma \Phi.
\]
One then associates to this operator colligation an operator-valued function of a complex variable
--- the characteristic function of the colligation --- which is
holomorphic outside of the spectrum of $A$:
\[
S(z) = I - i \Phi \left( A - zI\right)^{-1} \Phi^* \sigma.
\]
This function is obtained from the following time invariant i/s/o system associated to the colligation:
\begin{equation*}
\begin{split}
& i f^{\prime}(t) + A f(t) = \Phi^* \sigma u(t), \\
& y(t) = u(t) - i \Phi f(t).
\end{split}
\end{equation*}
Here $f \colon \RR \to \cH$ is an absolutely continuous function and $u,y \colon \RR \to \cE$ are square-integrable functions. In this paper we will be dealing with representations of Lie algebras. It is more natural in this setting to consider non-skewselfadjoint operators rather than non-selfadjoint operators. Therefore the colligation condition in Definition \ref{vessel-def} is modified accordingly. One can pass from one setting to the other by replacing the representation $\rho$ with $\frac{1}{i} \rho$.

Such systems possess a frequency domain theory, namely we can pass to waves using the substitution $u(t) = e^{i z t} u_0$ and $f(t) = e^{i z t} f_0$, where $u_0 \in \cE$ and $f_0 \in \cH$ are fixed vectors. This shows that the characteristic function is in fact the transfer function of the frequency domain system obtained. See also the work of de Branges--Rovnyak \cite{dBR1,dBR2} and of Sz. Nagy--Foias \cite{NF}. 
Multiplicative decompositions of the characteristic function are closely related to invariant subspaces of $A$.
An early success of this approach was the proof that
a quasi-nilpotent dissipative completely non-selfadjoint operator 
with a finite dimensional imaginary part, acting on a separable Hilbert space, 
is unitarily equivalent to a Voltera-type operator
(\cite{Liv1946} in the case of a one-dimensional imaginary part
and the discussion preceding Theorem 16 in \cite{BrLiv}).

Starting with the work \cite{Liv78,Liv79inv,Liv79wav}  of M. S. Liv\v{s}ic (see also \cite{LivYan}),
these methods were generalized to the case of several, say $d$, commuting operators;
see the book \cite{LKMV} and the papers \cite{V1,V3,BV}.
The role of an operator colligation is taken over by a commutative $d$-operator vessel which is associated to a conservative overdetermined linear input/state/output system on ${\mathbb R}^d$.
This system is equipped with compatibility conditions for the input and output signals, and the frequency domain analysis leads to the transfer function, called the joint characteristic function of the vessel, which is a map of certain sheaves on the so called discriminant variety of the vessel holomorphic outside of the joint spectrum of the operators. 
In particular, in \cite{V1} and \cite{V3} (see also \cite[Ch.\ 10--12]{LKMV}), the second author constructed triangular models for pairs of commuting operators 
with finite dimensional imaginary parts, effectively solving the inverse problem for commutative two-operator vessels.

In this paper we propose to generalize and apply the classical methods described above to the study of non skew-selfadjoint representations of real finite dimensional Lie algebras. To this end we generalize the notion of a colligation/vessel to the non-commutative setting and develop the theory of associated shift invariant conservative systems on the associated simply connected Lie group. We then develop the analogs of frequency domain theory and characteristic function. We then demonstrate an application of this theory to non skew-selfadjoint representations of the Lie algebra of the $ax+b$-group.

The time-dependent and/or non-commutative cases were also considered
by H. Gauchman \cite{G1,G2,G3,G4,G5} in a very general setting of connections on Hilbert bundles on smooth manifold, with an application to Lie groups and their Lie algebras,
by M. S. Liv\v{s}ic \cite{Liv01}, and by 
D. Alpay, A. Melnikov and the second author \cite{AMV2009,Mel2011};
see also \cite{LivVak74,Waks} and \cite[Ch.\ 2]{LKMV}.

\subsection*{Structure of the Paper}

In Section \ref{sec:vessel} we will lay down the foundations of the theory. We will define the Lie algebra operator vessel in Definition \ref{vessel-def} and show in Proposition \ref{prop-rep_embedding} that every representation of a finite dimensional Lie algebra can be embedded in a vessel. We will discuss various forms of equivalences of vessels to be used latter on. Then in Subsection \ref{subsec:system} we will define the associated shift-invariant system on the simply connected Lie group corresponding to our Lie algebra. We show that the system is conservative and define the adjoint system and an adjoint vessel. In Subsection \ref{subsec:compat} we note that the system we have defined is over-determined and thus requires compatibility conditions in the input and in the output. We prove on Proposition \ref{exit&uniq:systemeq_solut} that the system admits a unique solution for every admissible input and that the output satisfies the output conditions. We also tie these condition and the system to the continuous cohomology of the Lie group and the Lie algebra. Lastly, in Subsection \ref{subsec:principal} we define the principal subspace of the vessel and relate it to classical notion of the principal subspace of a colligation. We define the notion of minimality of a vessel and show in Theorem \ref{jft-determines-vessel} that under certain conditions, the classical characteristic functions can detect isomorphisms of representations of a Lie algebra.

In Section \ref{sec:freq} we develop the frequency domain analysis using representation theory of Lie groups. We replace the classical waves with irreducible representations of the simply connected Lie group. We then construct the frequency domain input and output compatibility conditions. We define the joint characteristic function mapping the admissible inputs to the admissible outputs and show that the existence of this function is connected to the Taylor joint spectrum of the given representation of the Lie algebra. We prove in Theorem \ref{prop-ad_nilpotent_independence} that if there exists an $ad$-nilpotent element $X$ in the complexification of our Lie algebra, such that a the image of $X$ under a certain tensor representation is invertible, then the joint characteristic function is defined and we have a concrete formula for it. It is classically known that in every complex Lie algebra there exist $ad$-nilpotent elements (see for example \cite{BenIsa77}), therefore this assumption is in fact only on the tensor representation. We use $ad$-nilpotent elements for technical reasons and because the Lie algebras that we are interested have an abundance of $ad$-nilpotent elements. For other uses of $ad$-nilpotent elements, see for example \cite{Seg67}.

Section \ref{sec:ax+b} contains a detailed example of the simplest non-abelian case of the Lie algebra of the $ax+b$-group, the group of orientation preserving affine transformations of the real line. We show that given the external data, the vessel is defined uniquely by the joint characteristic function which is a monodromy preserving mapping between solution spaces of two systems of singular differential equations in the complex plane (provided certain minimality assumptions hold). We also show in Theorem \ref{thm:classification_ax+b} that the joint characteristic function is an invariant of minimal $ax+b$ vessels with the same external data. The last theorem is an application of the methods developed in the paper for representations of the $ax+b$ group. 
{\em Assume that we are given two operators $A_1$ and $A_2$ on a Hilbert space $\cH$, such that $A_2$ is completely non skew-selfadjoint with finite dimensional real part;
assume that $(A_1 + A_1^*) \cH \subset (A_2 + A_2^*) \cH$ and $[A_1,A_2] = A_2$; then $\dim \cH < \infty$}. A reader that is mostly interested in the results of this section can skip Section \ref{sec:freq}, provided that he is willing to accept the singular ODEs used throughout Section \ref{sec:ax+b}. 

The first author thanks Eitan Sayag for fruitful conversations regarding various aspects of representation theory.

\addtocontents{toc}{\protect\setcounter{tocdepth}{10}}

\section{Lie Algebra Operator Vessels} \label{sec:vessel}

\subsection{Definitions and Basic Results} \label{subsec:defs}

In this section we will introduce the main objects of study of this paper, namely the Lie algebra operator vessels and their associated systems of differential equations. Vessels were defined for two time-dependent operators in \cite[Sec.\ 2.2]{Liv1} and as second-order connection colligation \cite[Def.\ 3.1]{G2} and in particular \cite[Ex.\ 3.5]{G2}.

A few notations are in order. Let $\cH$ and $\cE$ be two Hilbert spaces, we will denote by $\mathcal{L}(\cH,\cE)$ the space of all bounded linear operators from $\cH$ to $\cE$. For simplicity we will denote $\cL(\cH,\cH)$ by $\cL(\cH)$.

We will denote Lie groups and their Lie algebras by capital and lowercase German letters, respectively. We fix $\fg$, a finite dimensional Lie algebra over $\RR$, and $\fG$, a simply connected Lie group corresponding to $\fg$. We will denote the Lie bracket on $\fg$ by $[\cdot,\cdot]$. We will denote as well by $\wedge^2 \fg = \fg \wedge \fg$ the second exterior power of $\fg$ and by $X \wedge Y$ the decomposable tensors.

\begin{dfn} \label{vessel-def}
A $\fg$-operator vessel is a collection:
\begin{equation} \label{eq:vessel}
\fV = \left(\cH,\cE,\gr,\Phi,\gs,\gg,\ggst\right).
\end{equation}
Here $\cH$ and $\cE$  are Hilbert spaces,
$\Phi \in \cL(\cH,\cE)$, and $\gr \colon \fg \to \cL(\cH)$,
$\gs \colon \fg \to \cL(\cE)$ and 
$\gg, \ggst \colon \bigwedge^2(\fg) \to \cL(\cE)$ are linear mappings such that the following hold:
\begin{itemize}
\item
$\gr$ is a representation of $\fg$ (see for example \cite{Nel59}),
i.e., for all $X,Y \in \fg$
\begin{equation} \label{eq:comm}
\gr([X,Y]).= [\gr(X),\gr(Y)] = \rho(X)\rho(Y) - \rho(Y)\rho(X).
\end{equation}
\item
For all $X \in \fg$, $\gs(X)^* = \gs(X)$,
and for all $X,Y \in \fg$,
\begin{align} \label{eq:2nd}
\begin{split}
& \gg(X \wedge Y) + \gg(X \wedge Y)^* +
\gs([X,Y]) = 0, \\
& \ggst(X \wedge Y) + \ggst(X \wedge Y)^* - \gs([X,Y]) = 0.
\end{split}
\end{align}
\item
The colligation condition: for all $X \in \fg$
\begin{equation} \label{eq:coll}
\gr(X) + \gr(X)^* = \Phi^* \gs(X) \Phi.
\end{equation}
\item
The input vessel condition: for all $X,Y \in \fg$
\begin{equation} \label{eq:input}
\gs(X) \Phi \gr(Y)^* - \gs(Y) \Phi \gr(X)^* = \gg(X \wedge Y) \Phi.
\end{equation}
\item
The output vessel condition: for all $X,Y \in \fg$
\begin{equation} \label{eq:output}
\gs(X) \Phi \gr(Y) - \gs(Y) \Phi \gr(X) = \ggst(X \wedge Y) \Phi.
\end{equation}
\item
The linkage condition: for all $X,Y \in \fg$
\begin{equation} \label{eq:linkage}
\ggst(X \wedge Y) + \gg(X \wedge Y) = \gs(X) \Phi \Phi^* \gs(Y) - \gs(Y) \Phi \Phi^* \gs(X).
\end{equation}
\end{itemize}

Alternatively, let $X_1,\ldots,X_l$ be a basis of $\fg$ and 
let $c_{kj}^m$ be the corresponding structure constants:
\begin{equation*}
[X_k,X_j] = \sum_m c_{kj}^m X_m.
\end{equation*}
Setting $A_k = \gr(X_k)$, $\gs_k = \gs(X_k)$, 
$\gg_{kj} = \gg(X_k \wedge X_j)$, $\ggst_{kj} = \ggst(X_k \wedge X_j)$,
we obtain the various vessel conditions in the basis dependent form:
\begin{align}
\label{eq:comm-basis}
& [A_k,A_j] = \sum_m c_{kj}^m A_m,\\
\label{eq:2nd-basis}
& \gg_{jk} + \gg^*_{jk} + \sum_m c_{jk}^m \gs_m = \ggst_{jk} + \ggst^*_{jk} - \sum_m c_{jk}^m \gs_m = 0,\\
\label{eq:coll-basis}
& A_k + A_k^* = \Phi^* \gs_k \Phi,\\
\label{eq:input-basis}
& \gs_k \Phi A_j^* - \gs_j \Phi A_k^* = \gg_{kj} \Phi,\\
\label{eq:output-basis}
& \gs_k \Phi A_j - \gs_j \Phi A_k = \ggst_{kj} \Phi,\\
\label{eq:linkage-basis}
& \ggst_{kj} + \gg_{kj} = \gs_k \Phi \Phi^* \gs_j - \gs_j \Phi \Phi^* \gs_k.
\end{align}
\end{dfn}

\begin{ex} \label{ex-commutative}
Let $\fg =\mathbb{R}^l$ and $\fG = \mathbb{R}^l$ then in particular the Lie bracket on $\fg$ is trivial. Hence if we consider the vessel conditions as described above we obtain that the structure constants are $c_{k j}^m = 0$. If we plug that into (\ref{eq:comm-basis})-(\ref{eq:2nd-basis}), we get:
\begin{align*}
\begin{split}
& [A_k,A_j] = 0 \\
& \gamma_{j k} = - \gamma_{j k}^* \\
& \gamma_{* j k} = - \gamma_{* j k}^* \\
\end{split}
\end{align*}
Which is exactly the fact that all of the $A_k$ commute with each other 
and that $\gamma_{j k}$ and $\gamma_{* j k}$ are selfadjoint. 
Thus we obtain that our conditions are exactly conditions \cite[(1.31) -- (1.34)]{V1}.
\end{ex}

\begin{ex} \label{ex:ax+b-vessel}
Let $\fg$ be the Lie algebra of the group of affine transformations of the real line. 
It is the only (up to isomorphism) non-commutative Lie algebra of dimension two. 
The algebra $\fg$ can be identified with the Lie subalgebra of $\fg \mathfrak{l}_2$:
\begin{equation*}
\fg \equiv \{ \gmatrix{p}{q} \colon p,q \in \mathbb{R} \}.
\end{equation*}

We can choose a basis, $\{X_1,X_2\}$, for $\fg$, such that $[X_1,X_2] = X_2$. 
Let $A_j = \rho(X_j)$, $\sigma_j = \sigma(X_j)$, $\gamma = \gamma(X_1 \wedge X_2)$ 
and $\gamma_* = \gamma_*(X_1 \wedge X_2)$. Then the conditions \eqref{eq:comm-basis} and \eqref{eq:2nd-basis} are:
\begin{align*}
\begin{split}
& [A_1, A_2] = A_2, \\
& \gamma + \gamma^* + \sigma_2 = 0, \\
& \gamma_* + \gamma_*^* - \sigma_2 = 0.
\end{split}
\end{align*}
\end{ex}

\begin{ex} \label{ex:heisenberg-vessel}
Let $\fg$ be the Heisenberg Lie algebra, i.e., the Lie algebra of the Heisenberg group, $H_n$. Then $\dim \fg = 2n + 1$ and we can choose a basis $\{X_1, \ldots, X_n, Y_1, \ldots Y_n, Z\}$, for $\fg$, such that:
\begin{align*}
\begin{split}
& [X_j, X_k] = [Y_j, Y_k] = 0, \mbox{ } 1 \leq j,k \leq n, \\
& [X_j,Y_k] = \left\{\begin{array}{ll} Z, & j = k \\ 0, & \mbox{else}\end{array}\right.,\\
& [X_j,Z] = [Y_j,Z] = 0, \mbox{ } 1 \leq j \leq n.
\end{split}
\end{align*}
The condition on $\rho$ is clear from those commutation relations. Furthermore we get that the values of $\gamma$ and $\gamma_*$ on all of the vectors of the associated basis for the exterior product are skew-selfadjoint, except for $\gamma(X_k \wedge Y_k)$ and $\gamma_*(X_k \wedge Y_k)$, for $1 \leq k \leq n$. For them we have:
\begin{equation*}
\gamma(X_k \wedge Y_k) + \gamma(X_k \wedge Y_k)^* + \sigma(Z) = \gamma_*(X_k \wedge Y_k) + \gamma_*(X_k \wedge Y_k)^* - \sigma(Z) = 0.
\end{equation*}

\end{ex}

\begin{dfn} \label{vessel-strict}
A Lie algebra operator vessel is called strict if: 
\begin{itemize}
\item $\Phi \cH = \cE$
\item $\bigcap_{X \in \fg} \ker \sigma(X) = 0$
\end{itemize}
\end{dfn}

\begin{dfn} \label{vessel-equiv}
Given two $\fg$-operator vessels with the same external data, $\fV_1 = \left(\cH_1,\cE,\rho_1,\Phi_1,\gs,\gg,\ggst\right)$ and $\fV_2 = \left(\cH_2,\cE,\rho_2,\Phi_2,\gs,\gg,\ggst\right)$, we say that $\fV_1$ is unitarily equivalent to $\fV_2$, if there exists an isometric isomorphism $U\colon \cH_1 \to \cH_2$, such that:
\begin{align*}
\begin{split}
& U \rho_1 U^{-1} = \rho_2 \\
& \Phi_2 U = \Phi_1
\end{split}
\end{align*}
\end{dfn}

\begin{dfn}
Given two $\fg$-operator vessels $\fV_1 = \left(\cH,\cE_1,\rho,\Phi_1,\gs_1,\gg_1,\gg_{*1}\right)$ and $\fV_2 = \left(\cH,\cE_2,\rho,\Phi_2,\gs_2,\gg_2,\gg_{*2}\right)$, with the same $\rho$, we say that $\fV_1$ is input/output equivalent to $\fV_2$ if there exists an isomorphism $T \colon \cE_1 \to \cE_2$, such that:
\begin{align*}
\begin{split}
& \sigma_1 = T^* \sigma_2 T, \\
& \gg_1 = T^* \gg_2 T, \\
& \gg_{1*} = T^* \gg_{*2} T,\\
& \Phi_2 = \Phi_1 T.
\end{split}
\end{align*}
\end{dfn}

The following proposition provides us with a constructive way of embedding a representation of a Lie algebra in a Lie algebra operator vessel.
\begin{prop} \label{prop-rep_embedding}
Given a representation $\gr \colon \fg \to \cL(\cH)$ of $\fg$, we define: 
\begin{align*}
\begin{split}
&\cE = \sum_{X \in \fg} (\gr(X) + \gr(X)^*)\cH, \\
&\Phi = P_{\cE}, \\
&\sigma(X) = \left(\gr(X) + \gr(X)^*\right)|_{\cE}, \\
&\gamma(X\wedge Y) =  \left(\gr(X)\gr(Y)^* - \gr(Y)\gr(X)^* - \gr([X,Y])^*\right)|_{\cE}, \\
&\gamma_{*}(X\wedge Y) = \left(\gr(X)^*\gr(Y) - \gr(Y)^*\gr(X) + \gr([X,Y])\right)|_{\cE}.
\end{split}
\end{align*}
Then the collection $\fV = \left(\cH,\cE,\gr,\Phi,\gs,\gg,\ggst\right)$ is a strict Lie algebra operator vessel.
\end{prop}
\begin{proof}
Obviously $\sigma$ and $\gamma$ are linear mapping from $\fg$ and $\wedge^2 \fg$, respectively, into $\cL(\cE)$. Now we must prove the vessel conditions. Quite clearly $\sigma$ is selfadjoint. Now the colligation condition follows immediately from the definitions and (\ref{eq:2nd}) follows from the fact that:
\begin{align*}
\begin{split}
& \gamma(X\wedge Y) + \gamma(X\wedge Y)^* = - \gr([X,Y])^* - \gr([X,Y]))|_{\cE} \\
& \gamma_{*}(X\wedge Y) + \gamma_{*}(X\wedge Y)^* = (\gr([X,Y]) + \gr([X,Y])^*)|_{\cE}
\end{split}
\end{align*} 
The input/output compatibility conditions follow easily as well, for example in the input case:
\begin{equation*}
\begin{split}
& \sigma(X)\Phi\gr(Y)^* - \sigma(Y)\Phi\gr(X)^* = \\ &
(\gr(X) + \gr(X)^*)\gr(Y)^* - (\gr(Y) + \gr(Y)^*)\gr(X)^* =\\ & (\gr(X)\gr(Y)^* + \gr(X)^*\gr(Y)^* - \gr(Y)\gr(X)^* - \gr(Y)^*\gr(X)^*) = \\ &
\gr(X)\gr(Y)^* - \gr(Y)\gr(X)^*) + \gr([X,Y])^* = \gamma(X\wedge Y)\Phi.
\end{split}
\end{equation*}
For the first equality we note that $\Phi^*$ is simply the embedding of $\cE$ in $\cH$.

As for the linkage condition we compute:
\begin{multline*}
\gs(X) \Phi \Phi^* \gs(Y) - \gs(Y) \Phi \Phi^* \gs(X) = (\gr(X) + \gr(X)^*)(\gr(Y) + \gr(Y)^*) - \\ (\gr(Y) + \gr(Y)^*)(\gr(X) + \gr(X)^*) =  \gr([X,Y]) + \gr(X)^* \gr(Y) - \gr(Y)^*\gr(X)  - \\ \gr([X,Y])^* + \gr(X) \gr(Y)^* - \gr(Y) \gr(X)^* = \gamma_*(X \wedge Y) + \gamma(X \wedge Y).
\end{multline*}
To see that the above vessel is strict we note that $\Phi$ is surjective by definition. Furthermore note that $\sum_{X \in \fg} \sigma(X)\cE = \cE$. If $e \in \bigcap_{X \in \fg} \ker \sigma(X)$, then for every $f \in \cE$ and $X \in \fg$ we have that:
\begin{equation*}
0 = \langle \sigma(X) e,f \rangle = \langle e, \sigma(X) f \rangle
\end{equation*}
Hence $e$ is orthogonal to all of $\cE$ and therefore $e = 0$.
\end{proof}

\begin{rem}
As one can see the input/output space of the above construction is the non-skewhermitian subspace of the operators $A_j = \rho(X_j)$, $j = 1,\ldots,l$.
\end{rem}

The following proposition shows that the above construction is the unique construction of a strict vessel for a given representation of $\fg$ up to input/output equivalence. Similar result has been proved in \cite[Prop.\ 8.1.1]{LKMV} in the case of strict colligations on Banach spaces.

\begin{prop}
Let $\fV = \left(\cH,\cE,\gr,\Phi,\gs,\gg,\ggst\right)$ be a strict vessel, then there exists an isomorphism of $\cE$ with the non-skewhermitian subspace of the representation, such that $\fV$ is input/output equivalent to the vessel constructed above.
\end{prop}
\begin{proof}
Let $\cE^{\prime} = \sum_{X \in \fg} (\gr(X) + \gr(X)^*)\cH$. If $f \in (\Phi^* \cE)^{\perp}$, we get that for every $X \in \fg$ and $g \cH$, we have: 
\begin{equation*}
\langle (\rho(X) + \rho(X)^*) f , g \rangle = i \langle \Phi^* \sigma(X) \Phi f, g \rangle = 0.
\end{equation*}
Hence in particular $f \in \cE^{\prime\perp}$. We conclude that $\ker \Phi = (\Phi^* \cE)^{\perp} \subseteq \cE^{\prime\perp}$. Therefore $\Phi$ is injective when restricted to $\cE^{\prime}$. 

Let $f \in \cE^{\prime\perp}$, then for every $X \in \fg$, $f \in \ker (\rho(X) + \rho(X)^*)$. To see it note that for every $g \in \cH$, we have:
\begin{equation*}
0 = \langle f , (\rho(X) + \rho(X)^*) g \rangle = \langle (\rho(X) + \rho(X)^*) f , g \rangle.
\end{equation*}
Hence $f \in \cap_{X \in \fg} \ker \Phi^* \sigma(X) \Phi$. Now since the vessel is strict $\Phi^*$ is injective and $\cap_{X \in \fg} \ker \sigma(X) = 0$. Therefore $f \in \ker \Phi$. We conclude that $\ker \Phi = \cE^{\prime\perp}$. Hence the restriction of $\Phi$ to $\cE^{\prime}$ is surjective.

We have proved that $T = \Phi|_{\cE^{\prime}}$ is an isomorphism. Let $\sigma^{\prime}(X) = (\gr(X) + \gr(X)^*)|_{\cE^{\prime}}$, then by the vessel conditions we have $T^* \sigma T = \sigma^{\prime}$. Consider the input vessel condition, \eqref{eq:input}, restricted to $\cE^{\prime}$ and apply $T^*$ to it to get:
\begin{equation*}
\begin{split}
T^* \gamma(X \wedge Y) T = (T^* \gs(X) \Phi \gr(Y)^* - T^* \gs(Y) \Phi \gr(X)^*)|_{\cE^{\prime}} \\
= ((\rho(X) + \rho(X)^*) \rho(Y)^* - (\rho(Y) + \rho(Y)^*) \rho(X)^*)|_{\cE^{\prime}} \\
= (\rho(X)\rho(Y)^* - \rho(Y)\rho(X)^* - \rho([X,Y])^*)|_{\cE^{\prime}}.
\end{split}
\end{equation*}
Here in the last equality we used \eqref{eq:comm}. Note that this is precisely the definition of $\gamma^{\prime}$ in the proposition above. Using \eqref{eq:output} one derives the same result for $\gamma_*$ and we are done.
\end{proof}

\begin{rem}
One can see from the proof that it is always the case that $\ker \Phi \subseteq \cE^{\prime}$. Hence if the image of $\Phi$ is closed (which is the case, for example, when $\dim \cE < \infty$), then one can write $\cE = \cE^0 \oplus \cE^{\dagger}$, where $\cE^{\dagger} = \Phi \cH$ and $\cE^0 = \ker \Phi^*$. Therefore in this case the vessel is an extension of a strict vessel.
\end{rem}

\begin{dfn} \label{dfn:puulback}
Assume that we have a $\fg$-operator vessel $\fV$ and that we have a Lie algebra map $\varphi \colon \fh \to \fg$. Then we define the pullback of $\fV$ along $\varphi$ as an $\fh$-operator vessel, given by:
\begin{equation*}
\varphi^* \fV = \left( \cH, \cE, \Phi, \rho \circ \varphi, \sigma \circ \varphi, \gamma \circ \wedge^2 \varphi, \gamma_* \circ \wedge^2 \varphi \right).
\end{equation*}
Here $\wedge^2 \varphi$ is the linear map induced by $\varphi$ on the exterior product, namely $\wedge^2 \varphi \colon \wedge^2 \fh \to \wedge^2 \fg$.
\end{dfn}

\subsection{Associated System} \label{subsec:system}
Similarly to \cite{BV}, a Lie algebra operator vessel corresponds to a left invariant linear system on the Lie group $\fG$.
\begin{dfn}
The associated system of the Lie algebra operator vessel as defined in Definition \ref{vessel-def} is the following system of differential equations on $\fG$:
\begin{equation} \label{eq:system}
\begin{split}
& X x + \gr(X) x = \Phi^* \gs(X) u,\\
& y = u - \Phi x.
\end{split}
\end{equation}
Here $x \colon \fG \to \cH$ is the state, $u \colon \fG \to \cE$ is the input, and $y \colon \fG \to \cE$ is the output, and $X$ and $Y$ are left invariant vector fields on $\fG$ (identified with elements of $\fg$).
In terms of a corresponding basis $X_1,\ldots,X_l$ for left invariant vector fields,
the system equations become:
\begin{align}
\label{eq:system-basis}
\begin{split}
& X_k(x) + A_k x = \Phi^* \gs_k u,\quad k=1,\ldots,l,\\
& y = u -  \Phi x.
\end{split}
\end{align}
\end{dfn}

\begin{ex}
Consider the commutative $l$-dimensional Lie algebra of Example \ref{ex-commutative}. Since $\fg$ is spanned by the vector fields $\pd{}{x_1},\ldots,\pd{}{x_l}$, we get the system of equations:
\begin{align*}
\begin{split}
& \pd{f}{x_k} + iA_k f = \Phi^* \sigma_k u \\
& y = u - \Phi f
\end{split}
\end{align*}
\end{ex}

\begin{ex} \label{ex:ax+b-system}
Let $\fG$, be the group of affine (orientation preserving) transformations of the real line. 
The group can be identified with a subgroup of $GL_2$ via:
\begin{equation*}
\fG \equiv \{ \Gmatrix{a}{b} \colon a \in \mathbb{R}_{> 0}, b \in \mathbb{R} \}.
\end{equation*}
The Lie algebra of this group is the Lie algebra described in Example \ref{ex:ax+b-vessel}. 
The left invariant vector fields are $X_1 = a \pda{}$ and $X_2 = a \pdb{}$.

We will use the same notations as in Example \ref{ex:ax+b-vessel} to write the system equations:
\begin{align*}
\begin{split}
& a \pda{x}(a,b) + A_1 x(a,b) = \Phi^* \sigma_1 u(a,b),\\
& a \pdb{x}(a,b) + A_2 x(a,b) = \Phi^* \sigma_2 u(a,b),\\
& y(a,b) = u(a,b) - \Phi x(a,b).
\end{split}
\end{align*}
Here $u$, $x$ and $y$ are appropriately-valued smooth functions.
\end{ex}

\begin{ex} \label{ex:heisenberg-system}
If $\fG = H_n$, the Heisenberg group of dimension $2n+1$, it is simply connected and its Lie algebra has been described in Example \ref{ex:heisenberg-vessel}. In what follows $j$ will run over the positive integers up to $n$.

The group $H_n$ can be identified with the group of matrices of the form:
\begin{equation*}
\begin{pmatrix}
1 & p_1 & \ldots & p_n & r \\
0 & 1 & 0 & \ldots & q_1 \\
\vdots & & \ddots & & \vdots \\
0 & \ldots & 0 & 1 & q_n \\
0 & \ldots & 0 & 0 & 1
\end{pmatrix}.
\end{equation*}

Then the left-invariant vector fields are: $X_j = \pd{}{p_j} - \frac{1}{2}q_j\pd{}{r}$, $Y_j = \pd{}{q_j} + \frac{1}{2}p_j\pd{}{r}$ and $Z = \pd{}{r}$. For functions on the group we will write the coordinates as $(p,q,r)$, where $p, q \in \mathbb{R}^n$ and $r \in \mathbb{R}$. The system equations are:
\begin{align*}
\begin{split}
& \pd{x}{p_j}(p,q,r) - q_j \pd{x}{r}(p,q,r) + \rho(X_j) x(p,q,r) = \Phi^* \sigma(X_j) u(p,q,r), \\
& \pd{x}{q_j}(p,q,r) - p_j \pd{x}{r}(p,q,r) + \rho(Y_j) x(p,q,r) = \Phi^* \sigma(Y_j) u(p,q,r), \\
& \pd{x}{r}(p,q,r) + \rho(Z) x(p,q,r) = \Phi^* \sigma(Z) u(p,q,r),\\
& y(p,q,r) = u(p,q,r) - \Phi x(p,q,r).
\end{split}
\end{align*}
\end{ex}

For every representation of a Lie algebra on a Hilbert space we have also the contragredient representation on the same Hilbert space, given by $\rho^*(X) = - \rho(X)^*$. Thus we are led to define the adjoint of a vessel (see \cite[Sec.\ 1.3]{BV} and \cite[Ch.\ 3.3]{LKMV} for the commutative case). 
\begin{dfn} \label{def-adjoint_vessel}
For a vessel:
\begin{equation*}
\fV = \left(\cH, \cE, \rho, \Phi, \sigma, \gamma, \gamma_*\right).
\end{equation*}
we define the adjoint veseel as:
\begin{equation*}
\fV^* = \left(\cH, \cE, -\rho^*,-\Phi,-\sigma,\gamma_{*}, \gamma\right).
\end{equation*}
\end{dfn}
We show next that $\fV^*$ is in fact a vessel. The colligation condition follows immediately from the colligation condition for $\fV$. Similarly 
the input and output vessel conditions follow from the definition. The linkage condition is immediate as well since we only interchanged $\gamma$ and $\gamma_*$.

Therefore $\fV^*$ is indeed a vessel. This vessel is called the adjoint vessel. Furthermore note that $\fV^{**} = \fV$, simply by construction.  

The left-invariant conservative linear system associated to $\fV^*$ has the form:
\begin{align} \label{eq:adj_system}
\begin{split}
& X \tilde{x} - \rho(X)^* \tilde{x} = \Phi^* \sigma(X) \tilde{u}, \\
& \tilde{y} = \tilde{u} + \Phi \tilde{x}.
\end{split}
\end{align}
We call this system the adjoint system of $\fV$.

The following proposition shows that the adjoint system is the original system with the input and output interchanged.

\begin{prop} \label{prop-adjoint_traj}
Let $(u,x,y)$ be a trajectory for the system \eqref{eq:system} associated to $\fV$. Then $(y,x,u)$ is a trajectory for the adjoint system \eqref{eq:adj_system}. 
\end{prop}
\begin{proof}
We compute:
\begin{equation*}
\begin{split}
& \Phi^* \sigma(X) y = \Phi^* \sigma(X) (u - \Phi x) = X x + \rho(X) x - \Phi^* \sigma(X) \Phi x = \\
& X x + \rho(X) x - (\rho(X) + \rho(X)^*) x = X x - \rho(X)^* x.
\end{split}
\end{equation*}
Furthermore, using again the fact that $y = u -  \Phi x$ we obtain that $ u = y + \Phi x$. Hence $(y,x,u)$ is in fact a trajectory for the adjoint system.
\end{proof}

The system \eqref{eq:system} is conservative. To compute the energy conservation law of the system we need the following lemma.

\begin{lem} \label{lem-energy}
Let $(u_1,x_1,y_1)$ and $(u_2,x_2,y_2)$ be trajectories of the system. Consider the smooth function $f \colon \fG \to \mathbb{C}$, defined by $f(g) = \langle x_1(g), x_2(g) \rangle_{\cH}$. Then for an arbitrary $X \in \fg$, we have:
\begin{equation*}
X f(g) = \langle \sigma(X) y_1(g) , y_2(g) \rangle_{\cE} - \langle \sigma(X) u_1(g), u_2(g) \rangle_{\cE}.
\end{equation*}
\end{lem}
\begin{proof}
We consider the trajectory  $(y_2,x_2,u_2)$ as a trajectory of the adjoint system and compute:
\begin{equation*}
\begin{split}
X f(g) & = X \langle x_1(g), x_2(g) \rangle_{\cH} = \langle X x_1(g) , x_2(g) \rangle_{\cH} + \langle x_1(g), X x_2(g) \rangle_{\cH} \\
& = \langle \Phi^* \sigma(X) u_1(g) - \rho(X) x_1(g), x_2(g) \rangle_{\cH} \\
& + \langle x_1(g), \Phi^* \sigma(X) y_2(g) + \rho(X)^* x_2(g) \rangle_{\cH} \\
& = \langle \sigma(X) u_1(g), \Phi x_2(g) \rangle_{\cE} + \langle \Phi^* x_1(g), \sigma(X) y_2(g) \rangle_{\cE} \\ & = \langle \sigma(X) u_1(g), y_2(g) - u_2(g) \rangle_{\cE} + \langle y_1(g) - u_1(g) , \sigma(X) y_2(g) \rangle_{\cE} \\ & = \langle \sigma(X) y_1(g) ,y_2(g) \rangle_{\cE} - \langle \sigma(X) u_1(g), u_2(g) \rangle_{\cE}.
\end{split}
\end{equation*}
\end{proof}

Hence we conclude that the system trajectories satisfy the following energy conservation law:
\begin{cor} \label{prop-energy}
Let $(u,x,y)$ be a trajectory of the system \eqref{eq:system}, then, for every $X \in \fg$ we have:
\begin{equation} \label{eq:energy}
X \langle x, x \rangle = \langle \gs(X) y, y \rangle - \langle \gs(X) u, u \rangle.
\end{equation}
In the basis-dependent form we have:
\begin{equation} \label{eq:energy-basis}
X_k \langle x, x \rangle = \langle \gs_k y, y \rangle - \langle \gs_k u, u \rangle,\quad k=1,\ldots,l.
\end{equation} 
\end{cor}

\begin{dfn} \label{dfn:rest_system}
Let $\fH \subset \fG$ be a simply connected Lie subgroup and let $\fh \subset \fg$ be the corresponding Lie subalgebra. 
The restriction of \eqref{eq:system} is the system of differential equations on $\fH$ obtained from \eqref{eq:system} 
by considering only $X \in \fh$ in the first equation of \eqref{eq:system} 
and letting $u$, $x$ and $y$ be functions on $\fH$.
\end{dfn}

The following is an immediate consequence of Definitions \ref{dfn:puulback} and \ref{dfn:rest_system}.

\begin{prop} \label{prop:rest_pullback}
Let $\fH \subset \fG$ be a simply connected Lie subgroup and let $\iota \colon \fH \to \fG$ be the inclusion map, 
so that $d\iota \colon \fh \to \fg$ is the inclusion map of $\fh$ in $\fg$. 
Let $\fV$ be a $\fg$-operator vessel, then the associated system of $\fV$ restricted to $\fH$ coincides with the associated system of $d\iota^* \fV$.
\end{prop}

\subsection{Compatibility Equations} \label{subsec:compat}

Following the lines of \cite{BV} and \cite{LKMV} we note that the system is overdetermined. Therefore the input of the system has to satisfy some compatibility conditions for the equations to have a solution.

\begin{prop} \label{exit&uniq:systemeq_solut}
The system equations \eqref{eq:system} are compatible for a simply connected Lie group $\fG$ if and only if the input signal $u$ satisfies
\begin{equation} \label{eq:comp-input-nonstrict}
\Phi^* \left(\gs(Y) X u - \gs(X) Y u + \gg(X \wedge Y) u\right) = 0
\end{equation}
for all left invariant vector fields $X$, $Y$ on $\fG$; the corresponding output signal $y$
then satisfies
\begin{equation} \label{eq:comp-output-nonstrict}
\Phi^* \left(\gs(Y) X y - \gs(X) Y y - \ggst(X \wedge Y) y\right) = 0.
\end{equation}
In the basis-dependent form we have:
\begin{align}
\label{eq:comp-input-basis-nonstrict}
& \Phi^* \left( \gs_k X_j u - \gs_j X_k u + \gg_{jk}\right) u = 0,\quad j,k=1,\ldots,l,\\
\label{eq:comp-output-basis-nonstrict}
& \Phi^* \left(\gs_k X_j y - \gs_j X_k y - \ggst_{jk}\right) y = 0,\quad j,k=1,\ldots,l.
\end{align}
\end{prop}

We will provide two proofs of this fact. The first is a direct proof similar to the proof in \cite{BV} and the second is a more high-level proof using the theory of continuous and differentiable cohomology developed in \cite{HochMos}.

\begin{proof}[First Proof]
We prove the ``only if'' part first. Assume the system equations \eqref{eq:system} are compatible, then we pick $u$, $x$ and $y$ solutions for the system and consider the following set of equations:
\begin{align}
& \label{prop3.1:eq1} \Phi^* \sigma(X) Y u = Y X x + \gr(X) Y x \\
& \label{prop3.1:eq2} \Phi^* \sigma(Y) X u = X Y x + \gr(Y) X x \\
& \label{prop3.1:eq3} \sigma([X,Y]) + \gamma(X\wedge Y) + \gamma(X\wedge Y)^* = 0.
\end{align}
The equation \eqref{prop3.1:eq1} comes form applying $Y$ to the first system equation for $X$ and \eqref{prop3.1:eq2} comes form applying $X$ to the first system equation for $Y$. The last equation is simply \eqref{eq:2nd}. Now we subtract \eqref{prop3.1:eq1} from \eqref{prop3.1:eq2} and obtain:
\begin{equation} \label{prop3.1:eq4}
\Phi^*(\sigma(Y) X - \sigma(X) Y) u = [X,Y] x + (\gr(Y) X - \gr(X) Y) x.
\end{equation}
We now consider the first system equation for $[X,Y]$:
\begin{equation*}
\Phi^*\sigma([X,Y]) u = [X,Y] x + \gr([X,Y]) x.
\end{equation*}
We plug in \eqref{prop3.1:eq3} and obtain:
\begin{equation} \label{prop3.1:eq5}
- \Phi^*(\gamma(X\wedge Y) + \gamma(X\wedge Y)^*) u = [X,Y] x + [\gr(X),\gr(Y)] x.
\end{equation}
Now we subtract \eqref{prop3.1:eq5} from \eqref{prop3.1:eq4} and get:
\begin{align} \label{prop3.1:eq6}
\begin{split}
& \Phi^*(\sigma(Y) X - \sigma(X) Y + \gamma(X\wedge Y)) u + \Phi^*\gamma(X\wedge Y)^* u \\
&  = (\gr(Y) X - \gr(X) Y - [\gr(X),\gr(Y)]) x.
\end{split}
\end{align}
Adjoining the input vessel equation and using the fact that $\sigma(X)$ is selfadjoint for every $X \in \fg$ and applying it to $u$, we get:
\begin{multline} \label{prop3.1:eq7}
 \Phi^*\gamma(X\wedge Y)^* u = (\gr(Y)\Phi^*\sigma(X) - \gr(X)\Phi^*\sigma(Y)) u \\ =
 \gr(Y)( X x + \gr(X) x) - \gr(X) ( Y x + \gr(Y) x) \\
 = \gr(Y) X x - \gr(X) Y x - [\gr(X),\gr(Y)] x.
\end{multline}
Now plugging \eqref{prop3.1:eq7} back into \eqref{prop3.1:eq6} we get: 
\[
\Phi^*(\sigma(Y) X - \sigma(X) Y + \gamma(X\wedge Y)) u = 0.
\]
Hence the non-strict input compatibility condition holds. We now show that the non-strict input compatibility condition implies the non-strict output compatibility condition. Consider the non-strict output compatibility condition \eqref{eq:comp-output-nonstrict}. We substitute the second of the system equations \eqref{eq:system} we obtain:
\begin{multline*}
\Phi^* \left(\gs(Y) X (u - \Phi x) - \gs(X) Y (u - \Phi x) - \ggst(X \wedge Y) (u - \Phi x) \right) \\ =
\Phi^* \left(\gs(Y) X u - \gs(X) Y u + \ggst(X \wedge Y) u\right) \\
- \Phi^* \left(\gs(Y) \Phi X x - \gs(X) \Phi Y x - \ggst(X \wedge Y) \Phi x\right).
\end{multline*}
Now we plug \eqref{eq:output} into the second term and using the first system equation we obtain:
\begin{multline*}
\Phi^* \left(\gs(Y) \Phi X x - \gs(X)\Phi Y x  - \gs(X) \Phi \gr(Y) x + \gs(Y) \Phi \gr(X)x\right) \\
=  \Phi^* \left(\gs(Y)\Phi (X x + \gr(X)x) -\gs(X) \Phi (Y x + \gr(Y)x)\right) \\
= \Phi^* \left(\gs(Y)\Phi\Phi^* \gs(X) u - \gs(X) \Phi \Phi^* \gs(Y) u\right) \\
= - \Phi^* \left(\ggst(X \wedge Y) + \gg(X \wedge Y)\right) u.
\end{multline*}
The last equality is obtained via the linkage condition. Plugging it back into our original equation we get:
\begin{multline*}
\Phi^* \left(\gs(Y) X (u - \Phi x) - \gs(X) Y (u -  \Phi x) - \ggst(X \wedge Y) (u - \Phi x) \right) \\ = 
\Phi^* \left(\gs(Y) X u - \gs(X) Y u - \ggst(X \wedge Y) u\right) \\
+ \Phi^* \left(\ggst(X \wedge Y) + \gg(X \wedge Y) u\right)
= 0.
\end{multline*}
Since we have assumed that the non-strict input compatibility condition is satisfied.

Now for the ``if'' part. Let $u \in C^{\infty}(\fG,\cE)$ be a function satisfying the non-strict input/output compatibility conditions and we must show that the system \eqref{eq:system} is compatible. 

Following the proof of \cite[Lem.\ 9.1]{Nel59}, we note that since every $\rho(X)$ is a bounded operator on $\cH$, every vector in $\cH$ is analytic for $\rho(X)$. Hence we can define a map $E(\exp X) = e^{\rho(X)}$. from some neighbourhood of the identity in $\fG$ to $\cL(\cH)$. The discussion in \cite[Ch.\ 4.10]{Che99} shows that this map is in fact a local homomorphism (see also \cite[Thm.\ III.4.1]{B1}). Now since $\fG$ is simply connected we can extend $E$ to a homomorphism satisfying:
\begin{equation*}
(X E)(g) = E(g) \rho(X)
\end{equation*}

We define the following $\cH$-valued 1-form:
\begin{equation*}
(\omega_u)(X)(g) =  E(g) \Phi^* \sigma(X) u(g).
\end{equation*}

We prove now that $\omega_u$ is a closed form. Recall that in order to show that a $1$-form,$\theta$, is closed, it suffices to check that $d\theta(X,Y) = 0$, for every two left-invariant vector fields $X,Y \in \fg$. This follows from the fact that the left invariant vector fields trivialize the tangent bundle globally. Now we note that:
\begin{equation*}
(d\omega_u)(X,Y) = X \omega_u(Y) - Y \omega_u(X) - \omega_u([X,Y]).
\end{equation*}
Now let us compute:
\begin{multline*}
(X \omega_u(Y))(g) = E(g) \rho(X) \Phi^* \sigma(Y) u(g) + E(g) \Phi^*\sigma(Y) (X u)(g).
\end{multline*}
On the other hand for the commutator we have:
\begin{align*}
\begin{split}
\omega_u([X,Y])(g) & = E(g) \Phi^* \sigma([X,Y]) u(g) \\
& = - E(g) \Phi^* (\gamma(X \wedge Y) + \gamma(X \wedge Y)^*) u(g) \\
& = E(g) \Phi^* \gamma (X \wedge Y) u(g) \\ & - E(g) (\rho(Y) \Phi^* \sigma(X) - \rho(X) \Phi^* \sigma(Y)) u(g).
\end{split}
\end{align*}
Now plug the above result into the formula for the differential of $\omega_u$ and get:
\begin{multline*}
(d \omega_u)(X,Y)(g) = E(g) \rho(X) \Phi^* \sigma(Y) u(g) + E(g) \Phi^*\sigma(Y) (X u)(g) \\ - E(g) \rho(Y) \Phi^* \sigma(X) u(g) - E(g) \Phi^*\sigma(X) (Y u)(g) + E(g) \Phi^* \gamma (X \wedge Y) u(g) \\ + E(g) (\rho(Y) \Phi^* \sigma(X) - \rho(X) \Phi^* \sigma(Y)) u(g) = 0.
\end{multline*}
The last equality follows from the assumption that $u$ is an admissible input.

Since $\fG$ is simply connected its first deRham cohomology group is trivial. Thus there exists a function $f \colon \fG \to \cH$, such that $d f = \omega_u$.

We define another smooth function $F \colon \fG \to \cL(\cH)$ by $F(g) = E(g^{-1}) $. Let us compute the differential of this function:
\begin{equation*}
X(g) F = \dfrac{d}{dt} (E((g \exp(tX))^{-1}))|_{t=0} = - \rho(X) F(g).
\end{equation*}

Now we look at the smooth function $x \colon \fG \to \cH$, $x(e) = h \in \cH$ defined by:
\begin{equation*}
x(g) = F(g)(h + f(e) - f(g)).
\end{equation*}
Plug the function $x$ into the system and we obtain:
\begin{multline*}
(X x)(g) + \rho(X) x(g) = (X F)(g) (h + f(e) - f(g)) \\ - F(g) (X f)(g)  + \rho(X) x(g) = - \rho(X) F(g) (h + f(e) - f(g)) \\ + \Phi^* \sigma(X) u(g) + \rho(X) x(g) = \Phi^* \sigma(X) u(g).
\end{multline*}
Hence the system is compatible.
\end{proof}

\begin{rem}
The above formula is the non-commutative analogue of the variation of parameter formula described in \cite{BV}.
\end{rem}

\begin{rem}
Note that the input and output compatibility conditions for the adjoint vessel are interchanged.
\end{rem}

\begin{rem} \label{onlystrict}
The strict input/output compatibility conditions are:
\begin{align}
\label{eq:comp-input}
& \gs(Y) X u - \gs(X) Y u +  \gg(X \wedge Y) u &= 0,\\
\label{eq:comp-output}
& \gs(Y) X y - \gs(X) Y y - \ggst(X \wedge Y) y &= 0,
\end{align}
or in the basis-dependent form:
\begin{align}
\label{eq:comp-input-basis}
& \gs_k X_j u - \gs_j X_k u + \gg_{jk} u = 0,\quad j,k=1,\ldots,l,\\
\label{eq:comp-output-basis}
& \gs_k X_j y - \gs_j X_k y - \ggst_{jk} y = 0,\quad j,k=1,\ldots,l.
\end{align}
If an input $u$ satisfies the strict input compatibility conditions, then the output $y$ satisfies the strict output compatibility conditions.

In the case of strict vessels the strict and non-strict compatibility conditions coincide. 
\end{rem}

\begin{ex} 
In the case $\fg = \mathbb{R}^n$ and thus $\fG = \mathbb{R}^n$, we get the compatibility conditions described in \cite{BV} and \cite{LKMV}:
\begin{equation*}
\left(\sigma_k \pd{}{t_j} - \sigma_j \pd{}{t_k} + \gamma_{jk} \right) u(t) = 0.
\end{equation*}
Note the slight difference, both in \cite{BV} and \cite{LKMV} the operator $\gamma_{jk}$ is selfadjoint but it is multiplied by $i$, so in fact it is the same equation.
\end{ex}

\begin{ex} \label{ex:ax+b-compat}
Let $\fG$ be the $ax+b$ group, then using the notations of Examples \ref{ex:ax+b-vessel} and \ref{ex:ax+b-system}, we get the compatibility equation:
\begin{equation*}
\left(a \sigma_2 \pda{} - a \sigma_1 \pdb{} + \gamma \right) u(a,b) = 0.
\end{equation*}
\end{ex}

\begin{ex} \label{ex:heisenberg-compat}
Let $\fG = H_n$, then using the notations of Examples \ref{ex:heisenberg-vessel} and \ref{ex:heisenberg-system}, 
we get the following system of equations ( $1 \leq j,k \leq n$ ):
\begin{align*}
\begin{split}
& \left( \sigma(X_k) \pd{}{p_j} - \sigma(X_j) \pd{}{p_k} + \frac{1}{2}(q_k \sigma(X_j) - q_j \sigma(X_k)) \pd{}{r} + \gamma(X_j \wedge X_k) \right) u(p,q,r) = 0, \\
& \left( \sigma(Y_k) \pd{}{q_j} - \sigma(Y_j) \pd{}{q_k} + \frac{1}{2} (p_j \sigma(Y_k) - p_k \sigma(Y_j)) \pd{}{r} + \gamma(Y_j \wedge Y_k) \right) u(p,q,r) = 0, \\
& \left( \sigma(X_k) \pd{}{q_j} - \sigma(Y_j) \pd{}{p_k} + \frac{1}{2} (p_j \sigma(X_k) + q_k \sigma(Y_j)) \pd{}{r} + \gamma(Y_j \wedge X_k) \right) u(p,q,r) = 0, \\
& \left( \sigma(X_k) \pd{}{r} - \sigma(Z) \pd{}{p_k} + \frac{1}{2} q_k \sigma(Z) \pd{}{r} + \gamma(Z \wedge X_k) \right) u(p,q,r) = 0, \\
& \left( \sigma(Y_k) \pd{}{r} - \sigma(Z) \pd{}{q_k} - \frac{1}{2} p_k \sigma(Z) \pd{}{r} + \gamma(Z \wedge Y_k) \right) u(p,q,r) = 0.
\end{split}
\end{align*}
\end{ex}

In fact system compatibility is a cohomological property. One can prove the above proposition in the following fashion.

\begin{proof}[Second Proof of Proposition \ref{exit&uniq:systemeq_solut}]
As we have seen above we can exponentiate the action $\rho$ of $\fg$ on $\cH$, to an action $E$ of $\fG$ on $\cH$.

Consider the $\fG$ module $C^{\infty}(\fG,\cH)$. We know \cite[Thm.\ 44.1]{Treves} that $C^{\infty}(\fG,\cH) \cong C^{\infty}(\fG)\bar{\otimes} \cH$ (here the tensor product is the completed projective tensor product). The isomorphism is given by sending a basic tensor $f \otimes v$ to the function $f(h) v$. The action of $\fG$ on $C^{\infty}(\fG,\cH)$ is given by $(g \cdot f)(h) = E(g) f(g^{-1} h)$.

The infinitesimal Lie algebra action on $C^{\infty}(\fG)$ is the action by derivations. Hence the action of the Lie algebra $\fg$ on $C^{\infty}(\fG)\bar{\otimes}\cH$ is given by:
\begin{equation*}
X (f \otimes v) = (X f) \otimes v + f \otimes \rho(X) v.
\end{equation*}

Now fix an admissible input $u \in C^{\infty}(\fG,\cE)$. We define a map $\Phi^*\sigma(\cdot) u \colon \fg \to C^{\infty}(\fG)\bar{\otimes}\cH$. Since $u$ is an admissible input we have:
\begin{align*}
\begin{split}
d (\Phi^* \sigma(\cdot) u)(X,Y) & = X (\Phi^* \sigma(Y) u) - Y (\Phi^* \sigma(X) u) - (\Phi^* \sigma([X,Y]) u) \\ & = \Phi^* \sigma(Y) X u + \rho(X) \Phi^* \sigma(Y) u - \Phi^* \sigma(X) Y u \\ & - \rho(Y) \Phi^* \sigma(X) u + \Phi^* \gamma(X \wedge Y) u + \Phi^* \gamma(X \wedge Y)^* u \\ & = \Phi^* \sigma(Y) X u - \Phi^* \sigma(X) Y u + \Phi^* \gamma(X \wedge Y) u = 0.
\end{split}
\end{align*}
Now as above, applying the vessel conditions and the fact that $u$ is admissible, tells us that $d (\Phi^* \sigma(\cdot) u) = 0$. Hence this map is a $1$-cocycle in $H^1(\fg,C^{\infty}(\fG) \bar{\otimes} \cH)$. 

Note that $\cH$ is an integrable $\fG$-module (cf. \cite{HochMos}), since the map $J \colon C_c^{\infty}(\fG,\cH) \to \cH$ is simply given by $J(f) = \int_{\fG} f d\mu_{\fG}$, where $\mu_{\fG}$ is the Haar measure on $\fG$ and the integral is the Bochner integral. 

Since $\fG$ is simply connected we know from \cite[Thm.\ 1]{Komy} that:
\begin{equation*}
H^1(\fg,C^{\infty}(\fG) \bar{\otimes} \cH) \cong H_{cont}^1(\fG,C^{\infty}(\fG) \bar{\otimes} \cH).
\end{equation*}

However by \cite[Lem.\ 5.2]{HochMos} or \cite[Lem.\ 5.2]{BorWal} we know that $C^{\infty}(\fG) \bar{\otimes} \cH$ is a continuously injective $\fG$-module. Hence $H^1(\fg,C^{\infty}(\fG) \bar{\otimes} \cH) = 0$. This implies that $\Phi^* \sigma(\cdot) u$ is a coboundary and thus there exists an element $x \in C^{\infty}(\fG) \bar{\otimes} \cH$, such that $d x = \Phi^* \sigma( \cdot) u$. Writing it out we get:
\begin{equation*}
X x -i \rho(X) x = \Phi^* \sigma(X) u.
\end{equation*}
In other words the system equations are compatible.
\end{proof}

\begin{rem}
In the first proof one clearly sees that given an input satisfying the compatibility equations, 
the state is defined up to a choice of an initial condition. 
In the second proof this fact is hidden in the kernel of the differential. The state is defined up to a shift by the kernel, which consists of constant functions.
\end{rem}

\subsection{Principal Subspace} \label{subsec:principal}
Next we show that the classical characteristic functions carry crucial information about our operator vessel.

First we give a few general definitions.
\begin{dfn}
Let $\fg$ be a Lie algebra and let $U(\fg)$ be the universal enveloping algebra of the complexification, $\cfg$. Then $ \rho$ extends to a representation of $U(\fg)$ on $\cH$. The principal space of a $\fg$-vessel is the closed linear envelope:
\begin{equation*}
\cP = \bigvee_{T \in U(\fg)} \rho(T) \Phi^* \cE
\end{equation*}
\end{dfn}
The principal subspace for a Lie algebra colligation was defined in \cite{Waks}.

\begin{rem}
By the Poincare-Birkhoff-Witt theorem $U(\fg)$ is generated by monomials $X_1^{k_1} X_2^{k_2} \cdots X_n^{k_n}$, where the $X_j$ form a basis for $\fg$. Hence in fact:
\begin{equation*}
\cP = \bigvee_{(k_1, \ldots, k_n) \in \mathbb{N}^n} \rho(X_1)^{k_1}\cdots \rho(X_n)^{k_n} \Phi^* \cE.
\end{equation*} 
\end{rem}

\begin{rem} \label{power-commutator}
By definition we have for every $X \in \fg$:
\begin{equation*}
\rho(X) \cP \subseteq \cP
\end{equation*}
\end{rem}

Now we prove the following proposition that shows that this case is not far from the commutative colligation case:
\begin{prop} \label{adjoint-principal}
Let $\cP^* = \bigvee_{T \in U(\fg)} \rho(T)^* \Phi^* \cE$. Then $\cP^*$ coincides with $\cP$.
\end{prop}
\begin{proof}
Let us show that $\cP^*$ is invariant under $\rho(X)$, for every $X \in \fg$. Then since $\Phi^* \cE \subseteq \cP^*$, we get that $\cP \subseteq \cP^*$. To see this note that by the colligation condition \eqref{eq:coll}, we have:
\begin{equation*}
\rho(X) = \Phi^* \sigma(X) \Phi - \rho(X)^*.
\end{equation*}
Hence for every $v \in \cP^*$:
\begin{equation*}
\rho(X) v = \Phi^* \sigma(X) \Phi v - \rho(X)^* v.
\end{equation*}
Note that $\Phi^* \sigma(X) \Phi v \in \Phi^* \cE$ and therefore $\rho(X) v \in \cP^*$. 

By symmetry of the argument we get that $\cP = \cP^*$.
\end{proof}

We now look at the subspace of $\cH$ orthogonal to $\cP$, namely $\cP^{\perp}$. 
By Remark \ref{power-commutator} and Proposition \ref{adjoint-principal}, 
the subspaces $\cP$ and $\cP^{\perp}$ are reducing subspaces for $\rho(X)$ for every $X \in \fg$. 
Note that for every $x \in \cP^{\perp}$, we have that $\langle x,\Phi^* e \rangle = 0$ for every $e \in \cE$, 
and therefore $\langle \Phi x, e \rangle = 0$, which implies that $\Phi x = 0$. 
In particular it follows from \eqref{eq:coll-basis} that the restriction of $\rho(X)$ to $\cP^{\perp}$ 
is skew-selfadjoint for every $X \in \fg$. 

\begin{dfn}
Given a $\fg$-vessel $\fV$, we will call $\fV$ minimal if $\cH = \cP$.
\end{dfn} 

The following proposition sheds additional light on the structure of the principal subspace:

\begin{prop} \label{principal-struct}
Assume that $\sigma(X)$ is invertible for some $X \in \fg$. Then $\cP = \bigvee_{k=0}^{\infty} \rho(X)^k \Phi^* \cE$.
\end{prop}
\begin{proof}
Denote by $\cP_X = \bigvee_{k=0}^{\infty} \rho(X)^k \Phi^* \cE$. Clearly $\cP_X \subseteq \cP$ and by definition $\Phi^* \cE \subseteq \cP_X$. Hence it suffices to show that $\cP_X$ is invariant under $\rho(Y)$, for all $Y \in \fg$. We show that in fact$\rho(Y) \rho(X)^k \Phi^* \cE \subseteq \cP_X$ by induction on $k$.

Take the adjoint of the input vessel condition, \eqref{eq:input}, to get:
\begin{equation*}
\rho(Y) \Phi^* \sigma(X) = \rho(X) \Phi^* \sigma(Y) + \Phi^* \gamma(X \wedge Y)^*.
\end{equation*}
Since $\sigma(X)$ is invertible we get that:
\begin{equation*}
\rho(Y) \Phi^* \cE \subseteq \rho(X) \Phi^* \cE + \Phi^* \cE \subseteq \cP_X.
\end{equation*}

Now we compute:
\begin{equation*}
\rho(Y) \rho(X)^k \Phi^* \cE = \rho(X) \rho(Y) \rho(X)^k-1 \Phi^* \cE + \rho([X,Y]) \rho(X)^k-1 \Phi^* \cE.
\end{equation*}
By induction we deduce that:
\begin{equation*}
\rho(Y) \rho(X)^k \Phi^* \cE \subseteq \rho(X) \cP_X + \cP_X = \cP_X.
\end{equation*}
\end{proof}

\begin{cor}
Combining Propositions \ref{adjoint-principal} and \ref{principal-struct} we get that if $\sigma(X)$ is invertible for some $X$, then:
\begin{equation*}
\cP = \bigvee_{k=0}^{\infty} \rho(X)^{*k} \Phi^* \cE.
\end{equation*}
\end{cor}

Recall that an operator colligation is a collection $\left(A, \cH,\cE,\Phi,\sigma \right)$, with $A$ a bounded operator on $\cH$, $\sigma$ a bounded operator on $\cE$ and $\Phi$ a bounded operator from $\cH$ to $\cE$, such that the colligation condition holds, namely:
\begin{equation*}
A - A^* = i \Phi^* \sigma \Phi.
\end{equation*}
The characteristic function of an operator colligation is a complex operator-valued function, defined by:
\begin{equation} \label{eq:charfunc_coll}
S(z) = I_{\cE} - i \Phi^* \left( A - z I_{\cH} \right)^{-1} \Phi \sigma.
\end{equation}
A colligation is called minimal if  $\cH = \bigvee_{k=0}^{\infty} A^k \Phi^* \cE$ 
(which is equivalent to $A$ being completely non-selfadjoint, i.e., 
having no non-trivial reducing subspace, so that the restriction of $A$ thereto is selfadjoint). 
A minimal colligation is defined up to a unitary equivalence by the characteristic function 
(\cite[Thm.\ 4.2]{LivYan}, also \cite{BrLiv} for the case when $\dim \cE < \infty$).

Now we are ready to prove the main theorem of this section:

\begin{thm} \label{jft-determines-vessel}
Let $\fV = \left(\cH_{v},\cE,\rho_{v},\Phi,\sigma,\gamma,\gamma_{*}\right)$ and \\ 
$\fU = \left(\cH_{u},\cE,\rho_u,\Psi,\sigma,\gamma,\gamma_{*}\right)$ be two minimal $\fg$-vessels with the same external data. Assume furthermore that $\sigma(X)$ is invertible for some $X \in \fg$. If the characteristic functions of $\fC_v = \left(\frac{1}{i}\rho_{v}(X),\cH_v,\cE,\Phi,\sigma(X)\right)$ and $\fC_u = \left(\frac{1}{i}\rho_u(X),\cH_u,\cE,\Psi,\sigma(X)\right)$ coincide in a neighborhood of infinity, then $\fV$ is unitarily equivalent to $\fU$. 
\end{thm}
\begin{proof}
The vessels colligation condition implies the colligation conditions for $\fC_v$ and $\fC_u$. By Proposition \ref{principal-struct} and the assumption that the vessels are minimal, we get that:
\begin{equation*}
\cH_{v} = \bigvee_{k=0}^{\infty} \rho_v(X)^k \Phi^* \cE \mbox{ and } \cH_{u} = \bigvee_{k=0}^{\infty} \rho_u(X)^k \Psi^* \cE.
\end{equation*}
By \cite{BrLiv} we get that there exists an isometry $U \colon \cH_v \to \cH_u$, such that $\rho_u(X) = U \rho_v(X) U^{-1}$ and $\Phi = \Psi U$.

Next note that for every $Y \in \fg$, by the input vessel condition \eqref{eq:input} we get:
\begin{equation*}
\rho_v(Y) \Phi^* \sigma(X) = \rho_v(X) \Phi^* \sigma(Y) + \Phi^* \gamma(X \wedge Y)^*.
\end{equation*}
Using the isometry $U$, we get:
\begin{equation*}
\rho_v(Y) U^{-1} \Psi^* \sigma(X) = \rho_v(X) U^{-1} \Psi^* \sigma(Y) + U^{-1} \Psi^* \gamma(X \wedge Y)^*.
\end{equation*}
Premultiplying the equation by $U$ gives us:
\begin{equation*}
U \rho_v(Y) U^{-1} \Psi^* \sigma(X) = \rho_u(X) \Psi^* \sigma(Y) + \Psi^* \gamma(X \wedge Y).
\end{equation*}
Now applying the input vessel condition on $\fU$, we get:
\begin{equation*}
U \rho_v(Y) U^{-1} \Psi^* \sigma(X) = \rho_u(Y) \Psi^* \sigma(X).
\end{equation*}
Recall that $\sigma(X)$ is invertible, therefore:
\begin{equation*}
U \rho_v(Y) U^{-1} \Psi^* = \rho_u(Y) \Psi^*.
\end{equation*}

Now note that since $\rho_u$ is a representation we get that if $w \in \Psi^* \cE$, then:
\begin{align*}
\begin{split}
\rho_u(Y) \rho_u(X) w & = \rho_u(X) \rho_u(Y) w + i \rho_u([Y,X])w \\
& = \rho_u(X) U \rho_v(Y) U^{-1} w + i U \rho_v([Y,X]) U^{-1} w \\ 
& = U \rho_v(Y) U^{-1} \rho_u(X) w.
\end{split}
\end{align*}
This implies that $U \rho_v(Y) U^{-1}$ and $\rho_u(Y)$ agree on the space $\rho_u(X) \Psi^* \cE$. 
Proceeding by induction we get that $U \rho_v(Y) U^{-1}$ and $\rho_u(Y)$ 
agree on $\rho_u(X)^k \Psi^* \cE$ for every $k \geq 0$. Since the closed linear envelope of those spaces is all of $\cH_u$, we deduce that $U \rho_v(Y) U^{-1} = \rho_u(Y)$, for every $Y \in \fg$.
\end{proof}

\begin{rem}
It has been proved in \cite{Waks} that a Lie algebra colligation is determined by the complete characteristic function of the colligation.
\end{rem}

\section{Frequency Domain Analysis} \label{sec:freq}

\subsection{Representations and the Frequency Domain} \label{subsec:reps}
The frequency domain theory is the analysis of the system at hand with respect to frequency rather than time. To analyze the associated system in the commutative case one starts with looking at particular system trajectories:
\begin{align*}
\begin{split}
& u(t) = e^{\langle \lambda, t \rangle} u_0, \\
& x(t) = e^{\langle \lambda, t \rangle} x_0, \\
& y(t) = e^{\langle \lambda, t \rangle} y_0.
\end{split}
\end{align*}
Here $\lambda = \left( \lambda_1,\ldots,\lambda_l\right)$ and $t = \left(t_1,\ldots,t_l\right)$.  Here $\lambda$ represents the frequency and $u_0$, $x_0$ and $y_0$ control the amplitude and phase of our signals. In this case we deal with the commutative Lie group, $\RR^{\ell}$, 
and the irreducible unitary representations are exactly $\pi_{\lambda} (t) = e^{\langle \lambda, t \rangle}$, where $\lambda \in i \RR^l$. 
This analogy leads us to consider representations of $\fG$ (see \cite{BV}, \cite{LKMV} and \cite{V1} for more details).
In Example \ref{ex:comm-char} below we will recover the commutative case. 

Let $\cH_{\pi}$ be a Hilbert space and $\pi \colon \fG \to \cL(\cH_{\pi})$ be a strongly continuous representation of $\fG$ on $\cH_{\pi}$.

A vector $\xi \in \cH_{\pi}$ is called smooth if the map $$\tilde{\xi}(x) = \pi(x) \xi$$ is a smooth map on $\fG$ with values in $\cH_{\pi}$. It turns out that all smooth vectors form a $\pi$-invariant subspace of $\cH_{\pi}$. We shall denote the space of smooth vectors by $\chpinf$. By a result of G\r{a}rding the space $\chpinf$ is dense in $\cH_{\pi}$. Furthermore $\chpinf$ can be identified with a closed subspace of $C^{\infty}(\fG,\cH_{\pi})$ and thus can be endowed with a topology, finer than the one induced from $\cH_{\pi}$, that turns $\chpinf$ into a Frechet space. The representation $\pi$ restricts to a continuous representation of $\fG$ on $\chpinf$. This representation is smooth since the smooth vectors of the restricted representation are all of $\chpinf$. See \cite[Ch.\ 4.4]{WarnerG1} for more details.

The Lie algebra $\fg$ of $\fG$ admits a representation on $\chpinf$ via:
\begin{equation*}
\pi(X) \xi = \lim_{t \to 0} \frac{\pi(\expmap(t X)) \xi - \xi}{t}.
\end{equation*}
The representation $\pi$ on $\chpinf$ is continuous. For every $\xi \in \chpinf$ and $X \in \fg$ we compute $X \tilde{\xi}$:
\begin{multline*}
(X \tilde{\xi})(x) = \dfrac{d}{dt} \tilde{\xi}(x \expmap(t X)) |_{t=0} = \dfrac{d}{dt} \pi(x) \pi(\expmap(t X))|_{t=0}  \xi \\ = \pi(x) \pi(X) \xi = \widetilde{\pi(X) \xi}(x).
\end{multline*} 
Since $\chpinf$ is a Frechet space, its strong dual, $\chpinf^{\prime}$, i.e., the space of continuous linear functionals endowed with the topology of uniform convergence on bounded sets, is a complete $DF$-space (\cite[Thm.\ IV.3.1]{Grthndk1} and \cite[Thm.\ 34.2]{Treves}). We denote by $\chpminf$ the space of distribution vectors, i.e., continuous anti-linear functionals on $\chpinf$. This space is canonically isomorphic to $\chpinf^{\prime}$ and hence is a complete $DF$-space as well. Furthermore by \cite[Prop\ 4.4.1.9]{WarnerG1}, the contragredient representation of $\pi$ is well defined on all of $\chpminf$ and is smooth. Therefore the Lie algebra, $\fg$, admits a representation on $\chpminf$ that we shall also denote by $\pi$. Similar computation as above shows that:
\begin{equation*}
(X \tilde{\xi})(x) = \widetilde{\pi(X) \xi}(x)
\end{equation*}
for all $\xi \in \chpminf$.

Let us fix some topological vector space $\cV$ and a continuous embedding of $\cV$ into $\chpinf$, such that the image of $\cV$ is dense in $\chpinf$ and is invariant under the action of $\fG$ and of $\fg$. Therefore the anti-linear dual of $\cV$, $\cV^*$, contains $\chpminf$. We identify $\cV$ with a subspace of $\chpinf$ and thus of $C^{\infty}(\fG,\cH_{\pi})$, yet keep in mind its finer topology.

Let us fix some completed topological tensor product and denote it by $\bar{\otimes}$. The following computations do not depend on the tensor product chosen. In general it will be either the inductive or the projective completed tensor product. In many cases in practice, these notions coincide.

We will consider trajectories of the form:
\begin{equation*}
\begin{split}
\bu(g) = (\pi(g) \otimes I_{\cE}) \bu_0,\\
\bx(g) = (\pi(g) \otimes I_{\cH}) \bx_0,\\
\by(g) = (\pi(g) \otimes I_{\cE}) \by_0.
\end{split}
\end{equation*} 
Here $\bu_0,\by_0 \in \cV^* \bar{\otimes} \cE$ and $\bx_0 \in \cV^* \bar{\otimes} \cH$.

Plugging those trajectories into the system equations we get:
\begin{align} \label{eq:freq_system}
\begin{split}
& (\pi(X)\otimes I_{\cH} + \ihptens \rho(X) )\bx_0 = (\ihptens \Phi^* \sigma(X)) \bu_0,\\
& \by_0 = \bu_0 - (\ihptens \Phi) \bx_0.
\end{split}
\end{align}
Then the frequency domain strict input/output compatibility conditions become:
\begin{align} \label{eq:freq_compat}
\begin{split}
& \left( \pi(X)\otimes\sigma(Y) - \pi(Y)\otimes\sigma(X) + \ihptens \gamma(X \wedge Y) \right) \bu_0 = 0, \\
& \left( \pi(X)\otimes\sigma(Y)  - \pi(Y)\otimes\sigma(X) - \ihptens \gamma_{*}(X \wedge Y)\right)\by_0 = 0.
\end{split}
\end{align}
Hence (similarly to the commutative case, see \cite{BV}) we are led to define the operators:
\begin{align*}
\begin{split}
& U(\pi,X,Y) = \pi(X)\otimes\sigma(Y) - \pi(Y)\otimes\sigma(X) + \ihptens \gamma(X \wedge Y), \\
& U_*(\pi,X,Y) = \pi(X)\otimes\sigma(Y) - \pi(Y)\otimes\sigma(X) - \ihptens \gamma_{*}(X \wedge Y).
\end{split}
\end{align*}
We define for each representation $\pi$ the spaces:
\begin{align}
& \label{input-space} \bcE(\pi) = \left\{ \bu_0 \in \cV^* \bar{\otimes} \cE \colon U(\pi,X,Y) \bu_0 = 0 \mbox{, } \forall X,Y \in \fg\right\}, \\
& \label{output-space} \bcE_{*}(\pi) = \left\{ \by_0 \in \cV^* \bar{\otimes} \cE \colon U_*(\pi,X,Y) \by_0 = 0 \mbox{, } \forall X,Y \in \fg\right\}.
\end{align}

\subsection{Characteristic Function} \label{subsec:char_func}

The classical characteristic function of an operator colligation was described in the Introduction. We saw that one can obtain the classical function, by considering the associated system and passing to the frequency domain. In this subsection we will proceed similarly in our non-commutative setting with appropriate adjustments introduced in the previous subsection.

The following short remark will be used extensively in the following discussion.

\begin{rem} \label{rem:jcf_defined}

Assume that the frequency domain state system equation admits a unique solution for every input 
that solves the frequency domain strict input
compatibility conditions, i.e., for every $\bu_0$ satisfying \eqref{eq:freq_compat} there exists a unique $\bx_0$ satisfying the first equation in \eqref{eq:freq_system} for all $X \in \fg$. Using the second system equation we can assign to each element $\bu_0 \in \bcE$ an element of $\by_0 \in \cV^* \bar{\otimes} \cE$ via:
\begin{equation*}
\by_0 = \bu_0 - (\ihptens \Phi) \bx_0.
\end{equation*}
Just as in Proposition \ref{exit&uniq:systemeq_solut} we have that:
\begin{equation*}
U_*(\pi,X,Y) \by_0 = 0.
\end{equation*}

\end{rem}

\begin{dfn}
Fix a representation $\pi$ of $\fG$. Assume that for every $\bu_0 \in \bcE(\pi)$, there exists a unique $\bx_0 \in \cV^* \bar{\otimes} \cH$ and thus a unique $\by_0 \in \cV^* \bar{\otimes} \cE$, that solves \eqref{eq:freq_system}; as just noticed it is necessarily the case that $\by_0 \in \bcE_*(\pi)$. Then we define the joint characteristic function:
\begin{equation*}
\begin{array}{cccc}
S(\pi) \colon & \bcE(\pi) & \to & \bcE_*(\pi) \\
& \bu_0 & \mapsto & \by_0.
\end{array}
\end{equation*}
We say that the joint characteristic function is defined at $\pi$ \footnote{Note that the definition depends on the choice of $\cV$, however in most case this choice offers itself naturally and in the following discussion will be fixed once and for all.}.
\end{dfn}
Next we will discuss two cases in which the joint characteristic function is defined at $\pi$.

Assume that $X \in \fg$ is such that $\pi(X)\otimes I_{\cH} + \ihptens \rho(X)$ is invertible. Then we have a solution to the first equation of \eqref{eq:freq_system} in the following form:
\begin{equation*}
\bx_0 = \left(\pi(X)\otimes I_{\cH} + \ihptens \rho(X)\right)^{-1} \left(\ihptens \Phi^* \sigma(X)\right) \bu_0
\end{equation*}
Substituting it into the second equation we get:
\begin{multline*}
\by_0 = (\ihptens I_{\cE} - \left(\ihptens \Phi\right) \\ \times \left(\pi(X)\otimes I_{\cH} + \ihptens \rho(X)\right)^{-1} \left(\ihptens \Phi^* \sigma(X)\right)) \bu_0.
\end{multline*}
This computation leads us to define the complete characteristic function of the vessel:
\begin{multline} \label{eq:ctf}
W(\pi,X) = \ihptens I_{\cE} - \left(\ihptens \Phi\right) \\ \times \left(\pi(X)\otimes I_{\cH} + \ihptens \rho(X)\right)^{-1} \left(\ihptens \Phi^* \sigma(X)\right).
\end{multline}
Note that $W$ maps a pair consisting of a representation of $\fG$ and an element of $\fg$ to an operator on $\cV^* \bar{\otimes} \cE$.

Assume now that:
\begin{itemize}
\item there exists an $X \in \fg$, such that $\pi(X)\otimes I_{\cH} + \ihptens \rho(X)$ is invertible;

\item for every $\bu_0 \in \bcE(\pi)$ there exists $\bx_0 \in \cV^* \bar{\otimes} \cH$, 
such that for every $Y \in \fg$ (whether the ``resolvent'' 
$\left(\pi(Y)\otimes I_{\cH} + \ihptens \rho(Y) \right)^{-1}$ exists or not) we have:
\begin{equation*}
\left(\pi(Y)\otimes I_{\cH} + \ihptens \rho(Y) \right) \bx_0 = \left(\ihptens \Phi^* \sigma(Y)\right) \bu_0.
\end{equation*} 
\end{itemize}
It follows that for every $\bu_0 \in \bcE(\pi)$, 
$\bx_0$ is uniquely defined and thus the joint characteristic function is defined at $\pi$ and 
$S(\pi) = W(\pi,X)|_{\bcE(\pi)}$, which is independent of the choice of $X$, such that the resolvent exists. 

Next we study some properties of the joint characteristic function in terms of the cohomology of the Lie algebra $\fg$. First we need a technical lemma.

\begin{lem} \label{phi*gamma}
We have the following identities for every $X,Y \in \fg$:
\begin{equation} \label{phi*gamma:eq1}
\Phi^* \gamma(X \wedge Y) + \rho(Y)\Phi^*\sigma(X) - \rho(X)\Phi^*\sigma(Y) + \Phi^* \sigma([X,Y]) = 0.
\end{equation}

\begin{multline} \label{phi*gamma:eq2}
\left(\ihptens \Phi^*\right) U(\pi,X,Y) \\ =  \left(\pi(Y) \otimes I_{\cH} + \ihptens \rho(Y)\right) \left(\ihptens \Phi^* \sigma(X)\right) \\ -
 \left(\pi(X) \otimes I_{\cH} + \ihptens \rho(X)\right) \left(\ihptens \Phi^* \sigma(Y)\right) \\ - \ihptens \Phi^* \sigma([X,Y]).
\end{multline}
\end{lem}
\begin{proof}
To prove the first assertion we take the adjoint of \eqref{eq:input} and get:
\begin{equation} \label{phi*gamma-1}
\Phi^* \gamma(X \wedge Y)^* = \rho(Y) \Phi^* \sigma(X) - \rho(X) \Phi^* \sigma(Y).
\end{equation}
Next we apply $\Phi^*$ to \eqref{eq:2nd} and get:
\begin{equation*}
\Phi^* \gamma(X \wedge Y) + \Phi^* \gamma(X \wedge Y)^* + \Phi^* \sigma([X,Y]) = 0.
\end{equation*}
Now we plug in \eqref{phi*gamma-1} to obtain the first assertion.

Now to prove the second assertion, we write out explicitly the left hand side of \eqref{phi*gamma:eq2}:
\begin{equation*}
LHS = \pi(X)\otimes\Phi^* \sigma(Y) - \pi(Y)\otimes\Phi^* \sigma(X) + i \ihptens \Phi^* \gamma(X \wedge Y).
\end{equation*}
Now we apply \eqref{phi*gamma:eq1} proved above to get:
\begin{multline*}
LHS = \pi(X)\otimes\Phi^* \sigma(Y) - \pi(Y)\otimes\Phi^* \sigma(X) \\ 
- \ihptens \rho(Y) \Phi^* \sigma(X) +  \ihptens \rho(X) \Phi^* \sigma(Y) \\ - \ihptens \Phi^* \sigma([X,Y]).
\end{multline*}
Collecting the terms we get the desired equality.
\end{proof}

Now recall that $M_{\pi} = \cV^* \bar{\otimes} \cH$ is a $\fg$-module via the action:
\begin{equation*}
X \cdot (\xi \otimes h) = \pi(X) \xi \otimes h - i \xi \otimes \rho(X) h. 
\end{equation*}
This action extends uniquely to an action of $\cfg$, the complexification of $\fg$, on $M_{\pi}$. Furthermore the linear map $\sigma$ can be extended uniquely to $\cfg$. We shall abuse the notation and denote the extension by $\sigma$ as well.

Consider the function $f_{\bu_0}(X) = (\ihptens (\Phi^* \sigma(X)))\bu_0$, for $\bu_0 \in \cV^* \bar{\otimes}\cE$. This is a map from $\cfg$ to $M_{\pi}$ and thus a $1$-cochain of the Lie algebra $\cfg$ with coefficients in $M_{\pi}$. Next we note that by Lemma \ref{phi*gamma}, we have that:
\begin{equation*}
d(f_{\bu_0})(X,Y) = (\ihptens \Phi^*) U(\pi,X,Y) \bu_0.
\end{equation*}
Therefore if $\bu_0 \in \bcE(\pi)$ then in particular $f_{\bu_0}$ is a $1$-cocycle \footnote{One can also define $\bcE(\pi)$ and $\bcE_*(\pi)$ using the non-strict vessel compatibility conditions, namely $(\ihptens \Phi^*) U(\pi,X,Y)\bu_0 = 0$ and $(\ihptens \Phi^*) U_*(\pi,X,Y)\bu_0 = 0$. In that case we have that $\bu_0 \in \bcE(\pi)$ if and only if $f_{\bu_0}$ is a $1$-cocycle. Furthermore, in case the vessel is strict the two definitions coincide}.

Let us assume that $H^0(\cfg,M_{\pi}) = H^1(\cfg,M_{\pi}) = 0$. Then for each $\bu_0 \in \bcE(\pi)$ there exists a unique $\bx_0 \in M_{\pi}$, such that $(d \bx_0)(X) = f_{\bu_0}(X)$, for every $X \in \cfg$. Let us expand this equality:
\begin{equation*}
(\pi(X) \otimes I_{\cH} + \ihptens \rho(X) ) \bx_0 = \ihptens \Phi^* \sigma(X) \bu_0.
\end{equation*}
This implies that there exists a unique solution, $\bx_0$, to the system equations. Hence we have a linear map that maps $\bcE(\pi)$ to $\bcE_*(\pi)$, as per Remark \ref{rem:jcf_defined}. Now if there exists $X \in \fg$ such that $\pi(X) \otimes I_{\cH} + \ihptens \rho(X)$ is invertible, then $W(\pi,X)|_{\bcE(\pi)}$ is precisely this linear map. Furthermore it is independent of the choice of $X$ that makes $\pi(X) \otimes I_{\cH} + \ihptens \rho(X)$ invertible. 

 Hence we have proved the following theorem:
\begin{thm} \label{koszul-two-term}
If $H^0(\fg,M_{\pi}) = H^1(\fg,M_{\pi}) = 0$, then the joint characteristic function is defined at $\pi$. Furthermore, if there exists $X \in \fg$, such that $\pi(X) \otimes I_{\cH} + \ihptens \rho(X)$ is invertible, then $W(\pi,X)|_{\bcE(\pi)}$ is independent of the choice of $X$ and $S(\pi) = W(\pi,X)|_{\bcE(\pi)}$.
\end{thm}

\begin{rem}
In the definition of the joint characteristic function we may replace $\fg$ by $\cfg$. Furthermore, we may replace $X \in \fg$ by $X \in \cfg$ in both the assumption and the conclusion of Theorem \ref{koszul-two-term}. This is a non-vacuous generalization, since even in the case of two commuting operators we can have that the resolvent exists for no $X \in \fg$ but it does exist for some $X \in \cfg$, see for example \cite[Ex.\ 6.1]{LKMV}.
\end{rem}

\begin{rem}
Following J. L. Taylor (\cite{Tay1972a}) we say that $\pi$ belongs to the resolvent set of $\rho$ if the Koszul cochain complex is exact for the module $M_{\pi}$. The above theorem shows that if $\pi$ is in the resolvent set for $\rho$ then the joint characteristic function is defined at $\pi$.
\end{rem}

\begin{rem} \label{rem:h0_must_die}
Note that if $H^0(\fg,M_{\pi}) \neq 0$ there can not exist a unique solution since the kernel of the differential will parametrize the different solutions if one exists. Therefore if the characteristic function is defined at $\pi$ then we must have that $H^0(\fg,M_{\pi}) = 0$.
\end{rem}

The statement of Theorem \ref{koszul-two-term} contains two conditions, the first one is the vanishing of the zeroth and first cohomology groups and the second one is the existence of the resolvent. The question is whether there exists a connection between the two conditions. It is well known that the first condition does not imply the second, see for example \cite[Ex.\ 6.2]{LKMV}. The following proposition will demonstrate that if there exists an $X \in \fg$ that is $\ad$-nilpotent and the resolvent exists for $X$, then the first condition is satisfied.

We consider the general case of a Lie algebra module.
\begin{prop} \label{prop:ad-nilpotent-cohom}
 Let $M$ be a $\cfg$-module, via some representation $\nu$. Assume that there exists an $ad$-nilpotent element $X \in \cfg$, such that $\nu(X)$ is invertible. Then $H^0(\fg,M) = H^1(\fg,M) = 0$.
\end{prop}
\begin{proof}
First note that if $m \in M$ is such that $d m = 0$ it implies that $\nu(X) m = 0$, but since $\nu(X)$ is invertible we get that $m = 0$. Hence $H^0(\fg,M) = 0$. Another way to see this fact is by noting that $H^0(\fg,M) = M^{\fg}$, i.e., all elements annihilated by the Lie algebra action. Since $\nu(X)$ is invertible, there are no such.

Now let $f \colon \fg \to M$ be a $1$-cocycle, in other words we have that for every $Y,Z \in fg$:
\begin{equation*}
\nu(Y) f(Z) - \nu(Z) f(Y) - f([Y,Z]) = 0.
\end{equation*}
Then taking $Z = X$, we get that for every $Y \in \fg$, $Y \neq X$:
\begin{equation} \label{eq:f(Y)}
f(Y) = \nu(X)^{-1} \left( \nu(Y) f(X) + f([X,Y]) \right).
\end{equation}
We define $m = \nu(X)^{-1} f(X)$ and show that $d m = f$, in other words that $\nu(Y) m = f(Y)$, for all $Y \in \fg$.

Note that for every $Y \in \fg$:
\begin{equation} \label{eq:comm-rel-1}
[\nu(X)^{-1}, \nu(Y)] = - \nu(X)^{-1} \nu([X,Y]) \nu(X)^{-1};
\end{equation}
this follows from the fact that $\nu$ is a representation of $\fg$ and that for two elements of an associative unital algebra $a$ and $b$, with $a$ a unit, we have:
\begin{equation*} 
a^{-1} [a,b] a^{-1} = a^{-1} (a b - b a) a^{-1} = -[ a^{-1},b].
\end{equation*}
Using \eqref{eq:comm-rel-1} we can rewrite \eqref{eq:f(Y)} as
\begin{equation*}
f(Y) = \nu(Y) m + \nu(X)^{-1}\left( f([X,Y]) - \nu([X,Y]) m \right).
\end{equation*}
Hence $f(Y) = \nu(Y)m$ if and only if $f([X,Y]) = \nu([X,Y]) m$.
Replacing $Y$ with $[X,Y]$ we get inductively that $f(Y) = \nu(Y) m$ 
if and only if $f(\operatorname{ad}_X^k(Y)) = \nu(\operatorname{ad}_X^k(Y)) m$. 
Recalling that $X$ is $ad$-nilpotent we get that the second statement is true trivially for $k$ large enough. 

We conclude that $f$ is in fact a $1$-coboundary, and since $f$ was an arbitrary $1$-cocycle we get that $H^1(\fg,M) = 0$.
\end{proof}

Note that combining Theorem \ref{koszul-two-term} and the above proposition we get:

\begin{thm} \label{prop-ad_nilpotent_independence}
Assume that there exists an $\ad$-nilpotent element $X \in \cfg$, such that 
$\pi(X) \otimes I_{\cH} + \ihptens \rho(X)$ is invertible. 
Then the joint characteristic function is defined at $\pi$ and $S(\pi) = W(\pi,Y)|_{\bcE(\pi}$, 
where the restriction is independent of the choice of $Y \in \cfg$, such that the resolvent $\left(\pi(Y) \otimes I_{\cH} + \ihptens \rho(Y)\right)^{-1}$ exists.
\end{thm}

\begin{rem} \label{rem:intertwining}
Theorem \ref{prop-ad_nilpotent_independence} implies that if there exists an $\ad$-nilpotent element $T \in \cfg$, such that the resolvent exits, then in particular $W(\pi,T)$ map $\bcE(\pi)$ to $\bcE_*(\pi)$. This fact can be proved directly using the following generalized intertwining equality for the Lie algebra vessels case (cf \cite[Thm.\ 8.4.2]{LKMV} for the commutative case):
\begin{multline} \label{eq:intertwining}
U_*(\pi,X,Y)W(\pi,Z) - C(\pi,X,Y,Z) - D(\pi,X,Y,Z) = U(\pi,X,Y) \\
+ B(\pi,Y)^* A(\pi,Z)^{-1}\left(\ihptens \Phi^*\right) U(\pi,Z,X) \\
+ B(\pi,X)^* A(\pi,Z)^{-1}\left(\ihptens \Phi^*\right) U(\pi,Y,Z).
\end{multline}
Here we write:
\begin{equation*}
A(\pi, X) = (\pi(X) \otimes I_{\cH} + \ihptens \rho(X)),\quad 
B(\pi,X) = \ihptens \Phi^* \sigma(X),
\end{equation*}
\begin{multline*}
C(\pi,X,Y,Z) = B(\pi,X)^* A(\pi,Z)^{-1} A(\pi,[Z,Y]) A(\pi,Z)^{-1} B(\pi,Z) \\
- B(\pi,Y)^* A(\pi,Z)^{-1} A(\pi,[Z,X]) A(\pi,Z)^{-1} B(\pi,Z),
\end{multline*}
\begin{multline*}
D(\pi,X,Y,Z) = i B(\pi,Y)^* A(\pi,Z)^{-1} B(\pi,[Z,X])\\
- i B(\pi,X)^* A(\pi,Z)^{-1} B(\pi,[Z,Y]).
\end{multline*}
The proof is a straightforward computation using the vessel conditions and is omitted. 
Now assume that there exists an $\ad$-nilpotent $T \in \cfg$, such that $A(\pi,T)$ is invertible. Then for every $X,Y \in \fg$ and every $Z \in \cfg$ such that $A(\pi,Z)$ is invertible, we have that:
\begin{equation} \label{eq:kernel_c+d}
\bcE(\pi) \subseteq \operatorname{Ker} \left( C(\pi,X,Y,Z) + D(\pi,X,Y,Z) \right).
\end{equation}
Indeed for every $\bu \in \bcE(\pi)$ let us set:
\begin{equation*}
\bx = A(\pi,T)^{-1} (\ihptens \Phi^* \sigma(T)) \bu.
\end{equation*}
By Proposition \ref{prop:ad-nilpotent-cohom} we know that $H^1(\fg,M_{\pi}) = 0$ and hence for every $X \in \fg$, we have:
\begin{equation*}
A(\pi,X) \bx = B(\pi,X) \bu.
\end{equation*}
Hence:
\begin{multline*}
C(\pi,X,Y,Z) \bu = B(\pi,X)^* A(\pi,Z)^{-1} A(\pi,[Z,Y]) \bx \\ - B(\pi,Y)^* A(\pi,Z)^{-1} A(\pi,[Z,X]) \bx \\
 = B(\pi,X)^* A(\pi,Z)^{-1} B(\pi,[Z,Y]) \bu \\ - B(\pi,Y)^* A(\pi,Z)^{-1} B(\pi,[Z,X]) \bu \\
 = - D(\pi,X,Y,Z) \bu.
\end{multline*}
Thus combining \eqref{eq:intertwining} and \eqref{eq:kernel_c+d} we conclude that 
if there exists an $\ad$-nilpotent element $T \in \cfg$ such that $A(\pi,T)$ is invertible,
then for every $Z \in \cfg$, such that the resolvent $A(\pi,Z)^{-1}$ exists, 
we have that $W(\pi,Z)$ maps $\bcE(\pi)$ into $\bcE_*(\pi)$.
\end{rem}

\begin{ex} \label{ex:comm-char}
Assume $\fg \cong \RR^n$ and thus $\fG \cong \RR^n$ as well. 
Take  $\pi$ to be a (not necessarily unitary) character of $\fG$, 
namely $\pi(t) = \pi_{\lambda}(t) = e^{i \langle \lambda, t \rangle}$, 
where $\lambda=(\lambda_1,\ldots,\lambda_n) \in \CC^n$;
if $X_j$ is the standard basis for $\fg$, then $\pi(X_j) = i \lambda_j$ as an operator on $\CC$. 
The representation $\pi$ is finite dimensional hence smooth. 
Now the system of frequency domain strict input compatibility conditions becomes ($1 \leq j < k \leq n$):
$$
\left( \lambda_j \sigma_k - \lambda_k \sigma_j + \gamma_{jk} \right) \bu_0 = 0;
$$
here $\sigma_j=\sigma(X_j)$, $\gamma_{jk} = \gamma(X_j \wedge X_k)$, and $\bu_0 \in \cE$.
Therefore 
$$
\bcE(\pi)=\bcE(\lambda) 
= \bigcap_{1 \leq j < k \leq n} \ker\left( \lambda_j \sigma_k - \lambda_k \sigma_j + \gamma_{jk} \right).
$$
Set $\rho(X_j) = A_j$, then for every $X  = \sum \xi_jX_j \in \cfg \cong \CC^n$, 
we have $\rho(X) = \sum_j \xi_j A_j$. Writing out the formula for the characteristic function we get:
$$
S(\pi) = S(\lambda) = \left.\left(I - \Phi \big(\sum_j \xi_j A_j + (\sum_j \xi_j\lambda_j) I\big)^{-1} 
\Phi^* \big( \sum_j \xi_j \sigma_j\big)\right)\right|_{\bcE(\lambda)}.
$$
The joint characteristic function is defined at $\lambda$ whenever there exists
$X$ such that the resolvent exists --- and the above restriction is independent of the choice of $X$ --- 
since the algebra is commutative and hence every $X \in \cfg$ is $\ad$-nilpotent. 

Thus we have recovered the classical frequency domain theory for commutative operator vessels as in \cite{LKMV}. Note that:
\begin{itemize}
	\item The joint characteristic function is defined at $\lambda$ whenever $\lambda$ is not in the Taylor joint spectrum of $A_1,\ldots,A_n$, though if $\cH$ is infinite-dimensional this does not imply that the resolvent exists for some $X$; for the discussion of some cases when it does see \cite[Ch.\ 5--6]{LKMV}.
	
	\item When $n \geq 2$ and $\dim \cE < \infty$, the space $\bcE(\lambda)$ is zero off the input discriminant variety in $\CC^n$ defined by the equations:
	\begin{equation*}
	\det\left( \lambda_j \sigma_k - \lambda_k \sigma_j + \gamma_{jk} \right) = 0, 
	\end{equation*}
	and similarly for the space $\bcE_*(\lambda)$, see \cite[Ch.\ 7]{LKMV} for details.
	
	\item When $n=1$, the joint characteristic functions is defined at $\lambda$ if and only if $\ker(A - \lambda I) = 0$ and $\sum_{k=0}^{\infty} A^k \Phi^* \sigma \cE \subset \operatorname{im}(A - \lambda I)$; if $\lambda \notin \operatorname{sp}(A)$ the joint characteristic function coincides with the characteristic function of the colligation $(\cH,\cE,A,\Phi,\sigma)$ as defined in Section \ref{subsec:principal}.
\end{itemize}
\end{ex}

We finish this section by
formulating the uniqueness part of the inverse problem for the joint characteristic function
in the general case, similarly to \cite[Sec.\ 10.1]{LKMV}. 
Given a Lie algebra $\fg$ and the corresponding simply connected Lie group $\fG$, 
{\em find a suitable class of representations $\Pi$, 
such that for every two minimal vessels $\fV$ and $\fU$ with the same external data and $S_{\fV}(\pi) = S_{\fU}(\pi)$ 
for every $\pi \in \Pi$ (in particular one is defined at $\pi$ if and only if the other one is), 
it follows that $\fV$ is unitarily equivalent to $\fU$}.

\subsection{Plancherel's Theorem} \label{subsec:plancherel}

Throughout this section we will assume that $\fG$ is type-$I$. To simplify the formulas we will also assume that $\fG$ is unimodular and $\dim \cE < \infty$. All of the computations in this section are formal and are only intended to provide additional motivation.

In system theory one can pass to the frequency domain of a given linear system in several ways. One of those ways is by considering special wave trajectories. This approach was discussed above. The other way is to apply the Fourier transform. We will provide a formal version of that approach in the current section. This will give additional motivation to the use of unitary representations in the following sections.

Let $f \in C^{\infty}_c(\fG)$ and let $\pi$ be a unitary representation of $\fG$. Then we define:
\begin{equation*}
\widehat{f}(\pi) = \int_{\fG} f(g) \pi(g^{-1})\, d \mu(g).
\end{equation*}
Here $\mu$ is the Haar measure on $\fG$. This is the Fourier transform of $f$ and it is an operator on $\cH_{\pi}$.

Now if $u \in C^{\infty}_c(\fG,\cE)$, we can define:
\begin{equation*}
\widehat{u}(\pi) = \int_{\fG} u(g) \otimes \pi(g^{-1})\, d \mu(g).
\end{equation*}
That defines an operator $\widehat{u}(\pi) \colon \cH_{\pi} \to \cH_{\pi} \bar{\otimes} \cE$, we can take any tensor product since by our assumption $\cE$ is nuclear. Furthermore by the discussion in \cite[Sec.\ 4.4.1]{WarnerG1} we know that $\chpinf$ is mapped to $\chpinf \bar{\otimes} \cE$ by $\widehat{u}(\pi)$. Same argument using \cite[Prop.\ A2.4.1]{WarnerG1} shows that in fact $\cV$ is mapped to $\cV \bar{\otimes} \cE$ by $\widehat{u}(\pi)$. The representation obtained from $\pi$ on $\cV^*$ is the contragredient representation of $\pi$. Note that this representation in fact coincides with $\pi$ on $\cH_{\pi}$, since $\pi$ is unitary.

The inverse Fourier transform is given by (see \cite[Thm.\ 7.44]{F1})
\begin{equation*}
u(g) = \int_{\widehat{\fG}} \tr(\pi(g) \otimes \widehat{u}(\pi))\, d\widehat{\mu}(\pi).
\end{equation*}  
Here $\widehat{\mu}$ is the Plancherel measure and the trace is the per-coordinate trace, which is defined almost everywhere. Notice that if $T \in \cL(\cE)$ and $X \in \fg$, then, provided that the differential operator commutes with integration:
\begin{equation*}
T \left( X u(g) \right) = \int_{\widehat{\fG}} \tr( (\pi(g) \otimes I_{\cE} ) (\pi(X) \otimes T) \widehat{u}(\pi)) \, d\widehat{\mu}(\pi).
\end{equation*}

Therefore taking the input compatibility conditions and computing formally we get:

\begin{multline*}
\sigma(Y) X u - \sigma(X) Y u + \gamma(X \wedge Y) u = \sigma(Y) X \int_{\widehat{\fG}} \operatorname{tr}(\pi(g)\otimes\widehat{u}(\pi))\, d\widehat{\mu}(\pi) \\
- \sigma(X) Y \int_{\widehat{\fG}} \operatorname{tr}(\pi(g)\otimes\widehat{u}(\pi))\, d\widehat{\mu}(\pi) + \gamma(X \wedge Y) \int_{\widehat{\fG}} \operatorname{tr}(\pi(g)\otimes\widehat{u}(\pi))\, d\widehat{\mu}(\pi) \\
= \int_{\widehat{\fG}} \operatorname{tr}((\pi(g)\otimes I_{\cE}) ((\pi(X) \otimes \sigma(Y)) \widehat{u} - (\pi(Y) \otimes \sigma(X)) \widehat{u} \\
+ (\ihptens \gamma(X \wedge Y)) \widehat{u})))\, d\widehat{\mu}(\pi).
\end{multline*}
Hence $u$ satisfies the input compatibility conditions if and only if $U(\pi,X,Y) \widehat{u}(\pi) = 0$, for every $X,Y \in \fg$, $\widehat{\mu}$-almost everywhere. We conclude that the image of the operator $\widehat{u}(\pi) \colon \cH_{\pi} \to \cH_{\pi} \bar{\otimes} \cE$ lies inside $\bcE(\pi)$, $\widehat{\mu}$-almost everywhere. This differs slightly from the presentation in the previous sections and provides another take on the theory.

Applying the same considerations to the output compatibility equations we get that the image of $\widehat{y}(\pi)$ lies in $\bcE_*(\pi)$, $\widehat{\mu}$-almost everywhere. Now plugging the inverse Fourier transforms of $u$, $x$ and $y$ into the system equations, we get that:
$$
\widehat{y}(\pi) = W(\pi,X) \widehat{u}(\pi)
$$
$\widehat{\mu}$-almost everywhere, for every $X \in \fg$, such that $W(\pi,X)$ is defined.

Throughout this section we have noted that the calculations performed are purely formal. Indeed one can not make them precise even in the case when $\fg = \RR^2$, since $\bcE(\lambda_1,\lambda_2) = 0$ outside of a curve, $C$, defined by the input/output compatibility equations. Since the Plancherel measure in this case is just the Lebesgue measure on the plane, we get that our input signals are zero. To amend this, one uses a modified Fourier transform along the curve
$C$ as explained in \cite[Sec.\ 2.2]{BV}. An interesting question is to generalize this construction to the Lie algebra operator vessel setting.

\section{Example: The ax+b Group} \label{sec:ax+b}

\subsection{Frequency Domain for the $ax+b$ Group} \label{subsec:freq_dom_ax+b}

Let $\fg$ be the two dimensional Lie algebra, spanned as a vector space by its Jordan-H\"{o}lder basis $X_1$ and $X_2$, with the bracket satisfying $[X_1,X_2]=X_2$, as in Example \ref{ex:ax+b-vessel}.

Let $\fG$ be the group described in Example \ref{ex:ax+b-system}. It is easy to see that $\fG$ is an exponentially solvable Lie group. The basis of left invariant vector fields on $\fG$ is given by $X_1 = a \pda{}$ and $X_2 = a \pdb{}$. We will use the basis-dependent form of the vessel conditions and system equations in the future.

Following \cite[Ex.\ 5.4.2.1]{WarnerG1} we note that there are, up to unitary equivalence, two infinite dimensional unitary irreducible representations of $\fG$ that support the Plancherel measure on the unitary dual. These two representations, $\pi_{+}$ and $\pi_{-}$, on $\cH_{\pi_{+}} = L^2((0,\infty))$ and $\cH_{\pi_{-}} = L^2((-\infty,0))$, respectively, are given by:
\begin{equation*}
\begin{split}
(\pi(a,b) f)(t) = \sqrt{a} e^{2 \pi i b t} f(a t). 
\end{split}
\end{equation*}
Here $\pi$ is either $\pi_{+}$ or $\pi_{-}$ and $t \geq 0$ or $t \leq 0$, respectively. Fix $\pi = \pi_{+}$, similar considerations apply to $\pi_{-}$.
One can check that the corresponding representation of $\fg$ on $\chpminfplus$ is:
\begin{equation*}
\pi(X_1) = \frac{1}{2} I + t \pdt{}, \mbox{ } \pi(X_2) = M_{2 \pi i t}.
\end{equation*}
Here $M_f$ denoted the multiplication operator by the function $f$. By \cite[Sec.\ 7]{Good1969} the smooth vectors of $\pi$ are precisely the smooth functions in $f \in L^2((0,\infty))$, such that for every pair of positive integers $m,n$, we have $t^m f^{(n)} \in L^2((0,\infty))$. In particular $C^{\infty}_c((0,\infty)) \subset \chpinfplus$. 

\begin{prop}
The inclusion $C^{\infty}_c((0,\infty)) \hookrightarrow \chpinfplus$ is continuous and the image is dense, hence in particular one has that $\chpinfplus^{\prime} \subset D^{\prime}((0,\infty))$ (the space of distribution on $(0,\infty)$).
\end{prop}
\begin{proof}
Since $C^{\infty}_c((0,\infty))$ is an $LF$-space, by \cite[Prop.\ 13.1]{Treves} it suffices to show that for each compact $K \subset (0,\infty)$, the inclusion map of $C_K((0,\infty))$, i.e., smooth functions with support in $K$, is continuous. The topology on $C_K((0,\infty))$ is generated by the following family of seminorms:
\begin{equation*}
p_n(f) = \sup_{x \in K} \left|\dfrac{d^n f}{d t^n}(x)\right|.
\end{equation*}
The topology on $\chpinfplus$ is generated by a family of seminorms:
\begin{equation*}
q_X(f) = \|\pi(X) f\|_{L^2},
\end{equation*}
for $X$ a monomial in $U(\fg)$.

By \cite[Prop.\ 7.7]{Treves} the inclusion is continuous if for every $q_X$ there exists a continuous seminorm, $p$, 
on $C_K((0,\infty))$, such that $q_X(f) \leq p(f)$.

For every $f \in C_K((0,\infty))$ we have that:
\begin{equation*}
q_X(f) = \| \pi(X) f\|_{L^2} \leq \mu(K) \sup_{x \in K} \left|(\pi(X) f)(x)\right|.
\end{equation*}
Here $\mu(K)$ is the Lebesgue measure of $K$ (which is finite, since $K$ is compact). Now note that:
\begin{equation*}
\pi(X) f = \sum_{j=0}^m a_j P_j(t) \dfrac{d^j f}{d t^j},
\end{equation*}
where $P_j$ are polynomials and $a_j$ are some complex coefficients. Since every polynomial is bounded on $K$, we get that:
\begin{equation*}
q_X(f) \leq \mu(K) \sup_{x \in K} \left|\sum_{j=0}^m a_j P_j(t) \dfrac{d^j f}{d t^j}\right| \leq \mu(K) \sum_{j=0}^m M_j \sup_{x \in K} \left|\dfrac{d^j f}{d t^j}\right|.
\end{equation*}
However the last expression defines a continuous seminorm on $C_K((0,\infty))$. Therefore the inclusion is continuous.

The density of the compactly supported smooth functions follows from \cite[Lem.\ A.1.3]{CorGrl90}.
More directly, 
let us fix $\varphi \in C_c^{\infty}((0,\infty))$ 
supported in $[1-\delta,1+\delta]$ and such that for every $t \in [1-\delta_0,1+\delta_0]$ we have $\varphi(t) = 1$ 
for some $0 < \delta_0 < \delta$. Let $f \in \cH_{\pi_+,\infty}$ and set $f_k(t) = \varphi(t^{1/k}) f(t)$, for $k > 1$. 
Repeated applications of the chain rule show first that 
$\left( t \dfrac{d}{dt} \right)^r (\varphi(t^{1/k}))$ is bounded for all $k>1$ and $r$, 
and second that
$\|t^q \left( t \dfrac{d}{dt} \right)^r(f_k - f)\|_{L^2} \underset{k\to\infty}\to 0$ that for all $r$ and $q$,
hence $\|\pi(X)(f_k-f)\|_{L^2} \underset{k\to\infty}{\to} 0$ for all $X \in U(\fg)$, and we are done.
\end{proof}

We set $\cV = C_c^{\infty}((0,\infty))$, hence $\cV^* = D^*((0,\infty))$ is the space of anti-linear distributions on $(0,\infty)$. In particular both $\cV$ and $\cV^*$ are nuclear.

Assume from now on that we have a $\fg$-vessel (that we shall also refer to as a $ax+b$-vessel)
$\fV=(\cH,\cE,A_1,A_2,\Phi,\sigma_1,\sigma_2,\gamma,\gamma_*)$ with
$\dim \cE < \infty$. 

Let $\bx_0 \in D^{*}((0,\infty),\cH) \cong D^{*}((0,\infty)) \bar{\otimes} \cH$ and $\bu_0,\by_0 \in   D^{*}((0,\infty),\cE) \cong D^{*}((0,\infty)) \bar{\otimes} \cE$.
The frequency domain system equations take the form:
\begin{align*}
\begin{split}
& \left(\frac{1}{2} I + t \pd{}{t} + A_1\right) \bx_0 = \Phi^*\sigma_1 \bu_0, \\
& \left(A_2 + 2 \pi i t\right) \bx_0  = \Phi^*\sigma_2 \bu_0,\\
& \by_0 = \bu_0 - i \Phi \bx_0.
\end{split}
\end{align*}
The following differential equation on $(0,\infty)$ is the frequency domain strict
input compatibility condition for our vessel:
\begin{equation} \label{axb:input_eq}
\left(\frac{1}{2} \sigma_2 + t \sigma_2 \dfrac{d}{dt} - 2 \pi i t \sigma_1 + \gamma \right) \bu_0 = 0.
\end{equation}
Similarly the frequency domain strict output compatibility condition is:
\begin{equation} \label{axb:output_eq}
\left(\frac{1}{2} \sigma_2 + t \sigma_2 \dfrac{d}{dt} - 2 \pi i t \sigma_1 - \gamma_{*} \right) \by_0 = 0.
\end{equation}

Furthermore we have that:
\begin{prop}
The operator $A_2$ is quasinilpotent.
\end{prop}
\begin{proof}
Consider $D = ad_{A_2}$ acting on $\cL(\cH)$. This is a derivation. Furthermore $D^2 A_1 = 0$. By the Kleinecke-Shirokov theorem \cite[Thm.\ 17.1]{BS1}, we get that $D A_1$ is quasinilpotent. However, $D A_1 = -A_2$ and we are done.
\end{proof}

\begin{cor} \label{cor:resolv_exists}
The operator $T = M_{2 \pi i t} \otimes I_{\cH} + I_{\cV^*} \otimes A_2$ is invertible.
\end{cor}
\begin{proof}

Following the discussion in \cite[Ex.\ A2.3.1]{WarnerG1}, 
each $\cH$-valued distribution on $(0,\infty)$, i.e., an element of $D^*((0,\infty),\cH)$, 
defines a unique anti-linear $\cH$-distribution on $(0,\infty)$, i.e., a continuous linear functional on $C_c^{\infty}((0,\infty),\cH)$. Thus we can think of the operator $T$ as the transpose of the operator $S = M_{-2 \pi i t} + A_2^*$ acting on $C_c^{\infty}((0,\infty),\cH)$. Since $A_2$ is quasinilpotent, so is $A_2^*$. Hence for every $t \in (0,\infty)$, the operator $A_2^* - 2 \pi i t I$ is invertible. Hence the operator $S$ is invertible. The inverse of $S$ is continuous and therefore takes bounded sets to bounded sets. Thus its transpose, which is the inverse of $T$, is continuous as well, by the definition of the topology on $\cH$-distributions and \cite[Prop.\ 19.5]{Treves}.
\end{proof}

Now note that $X_2$ is $ad$-nilpotent. Hence by Theorem \ref{prop-ad_nilpotent_independence} we know that the joint characteristic function is defined at $\pi$ and the proof of Corollary \ref{cor:resolv_exists} shows that it is given by the following expression:
\begin{equation*}
S(\pi) = I_{\cE} - \Phi \left(A_2 + 2 \pi i t\right)^{-1} \Phi^* \sigma_2.
\end{equation*}
The operator $S(\pi)$ acts on the space $\bcE(\pi) \subset D^*((0,\infty),\cE)$ of solutions 
to the frequency domain strict input compatibility equation \eqref{axb:input_eq} by multiplication; 
the result belongs to $\bcE_*(\pi)$, 
the space of solutions to the frequency domain strict output compatibility equation \eqref{axb:output_eq}. 
Here $\pi$ is either $\pi_{+}$ or $\pi_{-}$.

\subsection{Complexification} \label{subsec:complex}

Note that in fact any solution of the frequency domain strict compatibility equations 
\eqref{axb:input_eq} and \eqref{axb:output_eq} admits a multivalued analytic continuation 
to some punctured disc around the origin, see \cite[Thm.\ 4.1.1]{CL1} and \cite[Thm.\ 4.2.1]{CL1}. 
We can thus consider the complex differential equations:
\begin{align} \label{complex-compat}
\begin{split}
& \left( z \sigma_2 \dfrac{d}{dz} - 2 \pi i z \sigma_1 + i \tau \right) \bu_0 = 0,\\
& \left( z \sigma_2 \dfrac{d}{dz} - 2 \pi i z \sigma_1 - i \tau_{*} \right) \by_0 = 0.
\end{split}
\end{align}
Here $\tau$ and $\tau_{*}$ are the imaginary parts of $\gamma$ and $\gamma_{*}$ respectively,
and we have used the vessel condition \eqref{eq:2nd-basis}. 

Assume that $\sigma_2$ is invertible. 
By \cite[Thm.\ 4.2.1]{CL1}, we get that the fundamental solution matrix of the equation has the form: 
$\Psi(z) = z^{P} B(z)$, where $B(z)$ is entire and $P$ is some constant matrix. 
Clearly such a function is continuous on $(0,\infty)$ and hence locally integrable. 
It therefore defines a distribution on $(0,\infty)$.

\begin{thm} \label{complexification}
Multiplication by the following meromorphic operator-valued function maps solutions of the first equation of \eqref{complex-compat} to the solutions of the second:
\begin{equation*}
S(z) = I_{\cE} - \Phi \left(A_2 + 2 \pi i z \right)^{-1} \Phi^* \sigma_2.
\end{equation*}
\end{thm}
\begin{proof}
Restricting the equations to the positive real axis, we know this from the general theory developed above. Now the result follows by analytic continuation.
\end{proof}

\begin{rem}
The above theorem can be proved via straightforward computations using Lemma \ref{phi*gamma} and the vessel conditions.
\end{rem}

One can also consider the complexification of the equations from the point of view of representation theory. 
Recall that $\pi_{+}$ and $\pi_{-}$ are obtained by inducing unitary characters 
of the one parameter subgroup $\exp(t X_2)$, which happens to be normal. 
This group consists precisely of the elements of the form $(1,t)$. 
The unitary characters then take the form $e^{2 \pi i s t}$, for $s \in \mathbb{R}$ 
(for details see \cite[Ch.\ 6.7.1]{F1}). 
Consider now taking a character of the form $\chi_w(t) = e^{2 \pi i w t}$, for $w \in \mathbb{C}$.  
Those are precisely the generalized characters considered by Mackey in \cite{Mack}. 
We can thus identify the set of all those generalized characters with $\mathbb{C}$. 
The group acts on the set of all characters via $(a,b)\cdot w = w/a$. 
Hence the orbits of the action are rays $\arg w = \theta$. 
Let $\pi_w$ be the representation induced from $\chi_w$ in the sense of Aarnes \cite{Aarnes}. 
Namely, we consider the action of $\fG$ on the right on the space of distributions on the upper half-plane 
and tensor it with $\CC$ acted on the left by the character $\chi_w$. 
The tensor is taken as a tensor of a right $\exp(t X_2)$ module with a left module of the same group, 
and the action is the regular action on the left on distributions.
Then $\pi_w$ can be viewed as a representation on $C_c^{\infty}((0,\infty))$ and:
\begin{equation*}
(\pi_w(a,b) f)(r) = \sqrt{a} e^{2 \pi i w b r} f(a r).
\end{equation*}
We identify $C_c^{\infty}((0,\infty))$ 
with the space of compactly supported smooth functions on the ray $\{z = r e^{i \theta} \mid r > 0\}$ with the measure $dr$.

If $w_1$ and $w_2$ are in the same orbit, the induced representations are equivalent, 
indeed note that if $w_1 = a w_2$ for some $a > 0$, then $\pi_{w_1}(a,0) = \pi_{w_2}(a,0)$ 
and this operator intertwines $\pi_{w_1}$ and $\pi_{w_2}$. Thus we may assume that $w = e^{i \theta}$. 
Computing the representation $\pi_w$ of the Lie algebra, we get the strict input compatibility equation:
\begin{equation} \label{eq:ray_compat}
\left( r \sigma_2 \dfrac{d}{dr} - 2 \pi i e^{i \theta} r \sigma_1 + i \tau \right) \bu_0 = 0.
\end{equation} 

Now note that $\dfrac{d}{dr} = e^{i \theta} \dfrac{d}{dz}$, 
hence this equation coincides with the first equation of \eqref{complex-compat}. 
Thus if $\bu_0$ solves the input differential equation of \eqref{complex-compat}, 
then the restriction of $\bu_0$ to the ray $\arg z = \theta$ solves the differential equation \eqref{eq:ray_compat}. 
On the other hand, as we have stated above, 
the solutions of \eqref{eq:ray_compat} can be analytically continued to solutions of \eqref{complex-compat} 
on the punctured plane. 
The exact same considerations apply to the output differential equation.

When $\sigma_2$ is invertible, the equations \eqref{complex-compat} are singular ODEs, 
with a singularity of first kind at $0$. We denote by $\bcE$ the space of solutions of the input equation 
and by $\bcE_{*}$ the space of solutions of the output equation. 
Let $\Psi(z)$ and $\Psi_{*}(z)$ be the fundamental matrices of solutions 
for the input and the output differential equations, respectively. 
By \cite[Ch.\ 4.4]{CL1} both of these matrices are multivalued, hence we have that:
\begin{equation*}
\Psi(z e^{2 \pi i}) =\Psi(z) M \mbox{ and } \Psi_{*}(z e^{2 \pi i} = \Psi_{*}(z) M_{*}.
\end{equation*}
Here $M$ and $M_{*}$ are the representations of the monodromy operators 
in the bases described by $\Psi(z)$ and $\Psi_{*}(z)$ for $\bcE$ and $\bcE_{*}$, respectively. 
Let $\bS \colon \bcE \to \bcE_{*}$ be the map between the spaces of solutions given by the multiplication by $S(z)$. 
Applying $\bS$ to the columns of $\Psi(z)$ we get the following expression:
\begin{equation*}
S(z) \Psi(z) = \Psi_*(z) C.
\end{equation*}
Since $S(z)$ is single-valued, if we apply the monodromy operator at the output, we get that:
\begin{equation*}
S(z e^{2 \pi i}) \Psi(z e^{2 \pi i}) = S(z) \Psi(z) M = \Psi_*(z) C M.
\end{equation*}
On the other hand:
\begin{equation*}
\Psi_*(z e^{2 \pi i}) C = \Psi_*(z) M_* C.
\end{equation*}
Hence we have that $C M = M_* C$. In other words we have proved that:
\begin{prop}
The mapping $\bS$ intertwines the monodromy operators.
\end{prop}
We will call the mapping $\bS$ {\em the joint characteristic function} of the $ax+b$-vessel.

\begin{rem}
Note that a linear map between solution spaces of the differential equations \eqref{complex-compat}
that intertwines the monodromy operators 
is always given by the multiplication by a single-valued matrix function 
that is analytic on ${\mathbb C} \setminus \{0\}$. 
Indeed, given such a map $\bS$, let us choose fundamental matrices $\Psi(z)$ and $\Psi_{*}(z)$,
and write $[\bS]$ for the matrix representing $\bS$ with respect to the corresponding choice of bases. 
Then the multiplication by the following matrix coincides with $\bS$:
\begin{equation*}
S(z) = \Psi_{*}(z) [\bS] \Psi(z)^{-1}.
\end{equation*}
Since $\bS$ intertwines the monodromy operators it is immediate that $S(z)$ is single-valued.

This discussion implies also that once the differential equations are fixed, 
the matrix-valued function $S(z)$ is in fact determined by a constant matrix $[\bS]$. 
\end{rem}

Returning to the setting of Theorem \ref{complexification}, we see 
that the matrix-valued function $S(z)$ is precisely the classical characteristic function of 
the colligation $(i A_2,\cH,\cE,\Phi,\sigma_2)$ as defined in Section \ref{sec:principal} 
(up to scaling the variable $z$ by $2\pi$). 
From this fact and Theorem \ref{jft-determines-vessel} we get the following theorem:
\begin{thm} \label{thm:classification_ax+b}
Assume that $\sigma_2$ is invertible. 
Then every two minimal $ax+b$-vessels 
with the same external data and the same joint characteristic function are unitarily equivalent.
\end{thm}
This theorem leads to the natural question of characterizing those mappings between the spaces $\bcE$ and $\bcE_*$
that can be realised as the joint characteristic function of a $ax+b$-vessel with the corresponding external data.
The corresponding question for commutative two-operator vessels has been largely settled 
in \cite[Sections 10.5 and 11.2]{LKMV} and \cite{V3}.

The following lemma is needed to prove the main result of this section.
\begin{lem} 
Let $s_0(z),\ldots,s_m(z)$ be functions analytic and single valued in some punctured disc $D(0,\rho) \setminus \{0\}$. Assume that:
\begin{equation} \label{stupid}
\sum_{j=0}^m (\log z)^{j} s_j(z) = \sum_{j=0}^m z^{n_j} (\log z)^{j} g_j(z),
\end{equation}
where $n_l$ are integers and $g_0(z),\ldots,g_m(z)$ are functions analytic and single valued on $D(0,\rho)$. Then we have that $s_j(z) = z^{n_j} g_j(z)$, so in particular $s_j$ has at most a pole at $0$, for every $j=0,\ldots,m$.
\end{lem}
\begin{proof}

We prove it by induction on $m$. For $m=0$ it is clear.  Now from \eqref{stupid} we get that:
\begin{equation*}
(\log z)^m (s_m(z) - z^{n_m} g_m(z)) = \sum_{j=0}^{m-1} (\log z)^{j} (z^{n_j} g_j(z) - s_j(z)).
\end{equation*}

Substituting $z e^{2 \pi i}$ and subtracting the above equality we get that:
\begin{multline*}
(s_m(z) - z^{n_m} g_m(z)) \sum_{k=1}^m \binom{m}{k} (2 \pi i)^k (\log z)^{m-k} \\ = \sum_{j=1}^{m-1} \left(\sum_{k=1}^j \binom{j}{k} (2\pi i)^k (\log z)^{j-k} \right) \\ \times (z^{n_j} g_j(z) - s_j(z)) + (z^{n_0} g_0(z) - s_0(z)).
\end{multline*}
Note that the coefficient of $(\log z)^{m-1}$ is $2 \pi i m (s_m(z) - z^{n_m} g_m(z))$, so applying the induction hypothesis we get that $s_m(z) = z^{n_m} g_m(z)$. Hence we can use the induction hypothesis again to deduce the lemma.

\end{proof}

\begin{thm}
Assume that $\fV$ is a minimal $ax+b$-vessel and that $\sigma_2$ is invertible. Then $\dim \cH < \infty$. 
\end{thm}
\begin{proof}
Note that by \cite[Thm.\ 4.2.1]{CL1} a fundamental solution matrix of the input compatibility equation has the form $\Psi(z) = B(z) z^P$, where $B$ is analytic at $0$. Similarly at the output we get a fundamental matrix $\Psi_*(z) = B_*(z) z^{P_*}$. Hence the matrices representing the monodromy operators with respect to $\Psi(z)$ and $\Psi_*(z)$ are $M = e^{2 \pi i P}$ and $M_* = e^{2 \pi i P_*}$, respectively. Now let $C$ denote the matrix representing multiplication by $S(z)$. We have proved above that $M_* = C M C^{-1}$. Hence $M$ and $M_*$ have the same Jordan block structure. Additionally it implies that multiplication by $S(z)$ maps generalized eigenspaces of $M$ to generalized eigenspaces of $M_*$.

Assume without loss of generality that both $P$ and $P_*$ are in the Jordan normal form. Note that the generalized eigenspace of $M$ corresponding to an eigenvalue $e^{2 \pi i \lambda}$ is precisely the sum of all generalized eigenspaces of $P$, corresponding to eigenvalues differing from $\lambda$ by an integer. Let $V$ be such a generalized eigenspace, then $C V$ is a generalized eigenspace for $M_*$ corresponding to the same eigenvalue $e^{2 \pi i \lambda}$. Hence in particular, this space is the sum of all eigenspaces of $P_*$, corresponding to eigenvalues differing from $\lambda$ by an integer. 

Let $\lambda_1,\ldots,\lambda_r$ be a set of representatives of the equivalence classes of eigenvalues of $P$, with respect to the equivalence relation: $\lambda \sim \mu$ if $\lambda - \mu \in \ZZ$. Let us write $P_j$ and $P_{*j}$ for the submatrices of $P$ and $P_*$, respectively, consisting of all the Jordan blocks corresponding to the equivalence class of $\lambda_j$. We get the equation:
\[
S(z) B(z) \begin{pmatrix}
z^{P_1} & 0 & \ldots & 0 \\
0 & z^{P_2} & \ldots & 0 \\
& & \ddots & \\
0 & 0 & \ldots & z^{P_r} 
\end{pmatrix} = B_*(z) \begin{pmatrix}
z^{P_{*1}} C_1 & 0 & \ldots & 0 \\
0 & z^{P_{*2}} C_2 & \ldots & 0 \\
& & \ddots & \\
0 & 0 & \ldots & z^{P_{*r}} C_r 
\end{pmatrix}.
\]
Here $C_j$ are the corresponding diagonal blocks of $C$. If we decompose $S(z) B(z)$ and $B_*(z)$ into blocks of columns of sizes corresponding to the $P_j$'s, we get:
\begin{equation} \label{eq:block_SB}
[S(z) B(z)]_j z^{P_j} = [B_*(z)]_j z^{P_{*j}} C_j.
\end{equation}
Since every entry of $z^{P_j}$ and $z^{P_{*j}}$ is either of the form $z^{\lambda_j + n} \log^k(z)$, for $n$ and $k$ integers, or $0$, we can factor out $z^{\lambda_j}$ on both sides. We now apply the lemma to deduce that $S(z)$ has at most a pole at $0$. Indeed note that $B(z)$ is a matrix valued function analytic at $0$ and  $\det B(z) \not\equiv 0$. Therefore, if $S(z)$ had an essential singularity at $0$, so would $S(z) B(z)$, but then at least one entry of $[S(z) B(z)]_j$, for some $j$, would have had an essential singularity. However, applying the lemma to each entry of the equality \eqref{eq:block_SB} we get a contradiction; note that every entry of $[S(z) B(z)]_j$ appears at least once since $z^{P_j}$ is an upper-triangular matrix with non-zero diagonal elements.

Since $S(z)$ is in fact the characteristic function of $i A_2$ it is $\sigma_2$-inner on the lower half-plane and analytic at infinity (in fact $S(\infty) = I$). Note that by Proposition \ref{principal-struct} the minimality of the vessel is equivalent to the minimality of the colligation $(i A_2,\cH,\cE,\Phi,\sigma_2)$.

On the other hand we can apply \cite[Thm.\ 2.16]{AlpGoh} to obtain that $S(z)$ can be realized as the characteristic function of a colligation with a finite-dimensional state space. Since two colligations with the same characteristic function are unitarily equivalent, we get that $\dim \cH < \infty$. 
\end{proof}

\bibliographystyle{plain}
\bibliography{Lie_Algebra_Operator_Vessels}{}

\end{document}